\documentclass[namedate,webpdf,imanum]{arxiv}%

\graphicspath{{Fig/}}

\newtheorem{theorem}{Theorem}
\newtheorem{lemma}[theorem]{Lemma}

\newtheorem{definition}[theorem]{Definition}

\numberwithin{equation}{section}

\usepackage{bm}
\renewcommand{\boldsymbol}{\bm}

\newcommand{\bsgamma}{{\boldsymbol{\gamma}}}

\newcommand{\bsnu}{{\boldsymbol{\nu}}}

\newcommand{\bsb}{{\boldsymbol{b}}}

\newcommand{\bsm}{{\boldsymbol{m}}}

\newcommand{\bst}{{\boldsymbol{t}}}
\newcommand{\bsx}{{\boldsymbol{x}}}
\newcommand{\bsy}{{\boldsymbol{y}}}

\newcommand{\setu}{\mathrm{\mathfrak{u}}}

\newcommand{\supp}{\mathrm{supp}}

\newtheorem*{assumption}{Assumptions}

\newcounter{assumpenum}

\newcounter{exampleenum}

\usepackage{scrextend,natbib}

\usepackage{tikz,pgfplots}
\usetikzlibrary{positioning}

\usepackage{subfig,epsfig,microtype,needspace}

\allowdisplaybreaks
\renewcommand{\proofname}{Proof.}
\begin{document}

\DOI{DOI HERE}
\copyrightyear{2026}
\vol{00}
\pubyear{2026}
\access{Advance Access Publication Date: Day Month Year}
\appnotes{Paper}
\copyrightstatement{Published by Oxford University Press on behalf of the Institute of Mathematics and its Applications. All rights reserved.}
\firstpage{1}


\title[Uncertainty quantification for Gevrey regular random domain deformations]{Uncertainty quantification for stationary and time-dependent PDEs\\subject to Gevrey regular random domain deformations}

\author{Ana Djurdjevac
\address{\orgdiv{Mathematical Institute}, \orgname{University of Oxford}, \orgaddress{\street{Woodstock Road}, \state{Oxford}, \postcode{OX2 6GG}, \country{UK}}}}
\author{Vesa Kaarnioja
\address{\orgdiv{School of Engineering Sciences}, \orgname{LUT University}, \orgaddress{\street{P.O.~Box 20}, \postcode{53851} \state{Lappeenranta}, \country{Finland}}}}
\author{Claudia Schillings and Andr\'e-Alexander Zepernick*
\address{\orgdiv{Department of Mathematics and Computer Science}, \orgname{Free University of Berlin}, \orgaddress{\street{Arnimallee~6}, \postcode{14195} \state{Berlin}, \country{Germany}}}}

\authormark{A. Djurdjevac et al.}

\corresp[*]{Corresponding author: \href{email:a.zepernick@fu-berlin.de}{a.zepernick@fu-berlin.de}}

\received{Date}{0}{Year}
\revised{Date}{0}{Year}
\accepted{Date}{0}{Year}


\abstract{We study uncertainty quantification for partial differential equations subject to domain uncertainty. We parameterize the random domain using the model recently considered by Chernov and L\^{e} (2024) as well as Harbrecht, Schmidlin, and Schwab (2024)  in which the input random field is assumed to belong to a Gevrey smoothness class. This approach has the advantage of being substantially more general than models which assume a particular parametric representation of the input random field such as a Karhunen--Lo\`eve series expansion. We consider both the Poisson equation as well as the heat equation and design randomly shifted lattice quasi-Monte Carlo (QMC) cubature rules for the computation of the expected solution  under domain uncertainty. We show that these QMC rules exhibit dimension-independent, essentially linear cubature convergence rates in this framework. In addition, we complete the error analysis by taking into account the approximation errors incurred by dimension truncation of the random input field and finite element discretization. Numerical experiments are presented to confirm the theoretical rates.}


\maketitle

\section{Introduction}

Domain uncertainties arise frequently in numerous applications  due to factors such as restricted resolution and noise in measurements or manufacturing tolerances. In particular, various biological processes are naturally posed on uncertain domains. Inaccuracies in surface imaging such as microscopy, tomography, and in product manufacturing lead to the problem of describing a variety of different phenomena on randomly deformed domains, see, e.g.,~\cite{ShapeUncertaintyArrhythmias} and~\cite{FormingInudcedDamage}. In recent years, there has been increasing attention on the topic \citep{castrillon2016analytic,ChDjEl20,djurdjevac2021linear,dolz22,HPS16,hiptmair2018large,xiu2006numerical}, particularly from a numerical perspective.%

The natural first question that arises is how to represent an uncertain domain and the quantities of interest to be considered.
A common way of representing the random domain is to assume the existence of a random field $\boldsymbol V(\omega)$ subject to a probability space $(\Omega,\mathcal A,\mathbb P)$ that maps a given deterministic domain $D_{\rm ref} \subset\mathbb{R}^d$, $d\in\{1,2,3\}$, to a random domain $D(\omega)$ for $\omega\in\Omega$. %
This is the basis of the so-called \emph{domain mapping method} \citep{HPS16,xiu2006numerical} which we will also adopt in the present study: if $u$ denotes the solution to a partial differential equation (PDE) on a random domain $D(\omega)$, then the pullback solution is defined as the composition $\widehat u=u\circ \boldsymbol V(\omega)$. If the domain mapping $\boldsymbol V(\omega)$ is sufficiently smooth with respect to the spatial variable (e.g., it is a $\mathcal C^2$-diffeomorphism), then the pullback solution $\widehat u$ can be expressed as the solution to a transformed PDE problem posed on the reference domain, equipped with a random coefficient and a random source term taking into account the effect of the random deformation. This reformulation allows
us to deal with potentially large deformations and to apply numerous available numerical methods for solving equations with random
coefficients on a fixed domain.

A common alternative to the domain mapping method is the so-called perturbation method~\citep{harbrecht2008sparse}, where one starts with a perturbation field defined
on the boundary of a reference domain and a shape Taylor expansion with respect
to this field is used to express the solution of the considered equation. The main disadvantage
of this method is that it is applicable only for small perturbations. In contrast to the domain mapping method, one only needs to define the random transformation on the boundary $\partial D_{\rm ref}$ instead of the closure of the entire domain $\overline{D_{\rm ref}}$. However, the perturbation approach requires both the reference domain and random perturbations to be $\mathcal C^2$-smooth, while the domain mapping method allows the reference domain $D_{\rm ref}$ to be Lipschitz with only the random perturbations $\boldsymbol V(\omega)$ required to be $\mathcal C^2$-smooth.

 For the computation of the expected solution, we will use lattice quasi-Monte Carlo (QMC) cubature rules. QMC methods are a popular method for assessing the response statistics of PDEs with random inputs, exhibiting dimension-independent, faster-than-Monte Carlo cubature convergence rates under moderate assumptions placed on the input random field~\citep{log,log2,log4,kuonuyenssurvey,KuoNuyens2018,kss12}. 

In a recent pair of articles~\cite{Chernov23a,Chernov23b} demonstrated that the input random field can be generalized beyond those considered in the previously existing literature while still retaining dimension-independent QMC convergence rates for the PDE response. They showed that it is enough to assume that the input random field satisfies a certain parametric regularity bound which ensures that the PDE response belong to the so-called Gevrey smoothness class. For details on Gevrey class functions and their properties we refer to \cite{RodinoGevrey}. Moreover, \cite{schmidlin24} showed that Gevrey regularity of parametric inputs is inherited by the solutions to a very general class of operator equations. We remark that some of the parametric regularity analysis carried out in~\cite{schmidlin24} may be applied to the stationary PDE problem considered in our work. However, since the focus in~\cite{schmidlin24} is not the application of numerical methods, certain constants appearing in their parametric regularity estimates are not tracked explicitly. In contrast, since our goal is the design of constructible QMC cubatures for the computation of response statistics, it is necessary to make the dependence of the parametric regularity bounds on the model parameters explicit in our analysis. In addition, we consider the parabolic case which is not treated in~\cite{schmidlin24}. In general, domain uncertainty quantification has focused on stationary problems in the existing literature. Nevertheless, there has been recent interest on quantifying domain uncertainty for parabolic PDEs~\citep{castrillon21,djurdjevac2021linear}.

Under the assumption of Gevrey regularity imposed on the random domain mapping, we can prove that the solutions to both, the Poisson equation and the heat equation on a random domain, belong to a Gevrey class. After truncating the infinite-dimensional input vector field $\boldsymbol V$ to $s$ dimensions, we can develop tailored QMC point sets achieving essentially linear cubature convergence rates independently of the stochastic dimension $s$ of the problem. The cubature convergence rates will be assessed in a series of numerical experiments. In addition, we shall also briefly address the other sources of discretization error stemming from dimension truncation of the input random field as well as finite element approximation of the PDE solution. Together they yield an overall error estimate with explicit decay rates with respect to the truncation dimension, spatial and temporal refinement levels, and the QMC convergence rate.

\paragraph{Contributions.} This paper presents the following key contributions:
\begin{itemize}
\item The Gevrey class accommodates very general parameterizations of domain uncertainty. Moreover, in addition to the Poisson equation, we also study domain uncertainty quantification for a \emph{parabolic} PDE.
\item We develop a novel parametric regularity analysis for these problems, which is used to design tailored \emph{QMC cubature rules} with essentially linear convergence rates \emph{independently of the truncation dimension}.
\item We consider the approximation errors incurred by dimension truncation of the random domain mapping and finite element discretization.
\item Numerical experiments are presented which confirm the sharpness of our results.
\end{itemize}

\paragraph{Outline of the paper.} This paper is organized as follows. The elliptic and parabolic PDE problems as well as the Gevrey regular parameterization of domain uncertainty are introduced in Section~\ref{sec:paramreg} alongside the definitions and main assumptions required for the subsequent analysis. In Section~\ref{sec:qmc}, we briefly review the basic properties of randomly shifted rank-1 lattice rules. The parametric regularity analysis for the considered model problems is carried out in Section~\ref{sec:4}. The corresponding error rates for dimension truncation, finite element discretization, and QMC integration are stated in Section~\ref{sec:error}. Numerical experiments are presented to assess the QMC convergence rate for both the stationary and time-dependent cases in Section~\ref{sec:numex}. Some conclusions on our findings are drawn in Section~\ref{sec:conclusions}. The appendix covers results on technical recurrence relations and PDE regularity used in our analysis.

\section{Notations, preliminaries, and assumptions}\label{sec:paramreg}

Let $U:=[-\frac12,\frac12]^{\mathbb N}$ be the set of parameters. We define the set of finitely-supported multi-indices by
\begin{align*}
  \mathscr F:=\bigg\{\bsnu\in\mathbb N_0^{\mathbb N}\mid |\bsnu|:=\sum_{j\geq 1}\nu_j<\infty\bigg\}.  
\end{align*}
 Let $\boldsymbol m,\bsnu\in\mathscr F$ be multi-indices and let $\boldsymbol b=(b_j)_{j\geq 1}$ be a sequence of real numbers. The support of a multi-index $\bsnu$ is defined as
 \begin{align*}
  {\rm supp}(\bsnu):=\{j\in\mathbb N\mid \nu_j\neq 0\}.   
 \end{align*}
 In what follows, we shall use the following multi-index notations:
$$
\boldsymbol m\leq\bsnu\quad\Leftrightarrow\quad m_j\leq \nu_j\quad\text{for all}~j\in\mathbb N,$$
$$\binom{\bsnu}{\boldsymbol m}=\prod_{j\geq 1}\binom{\nu_j}{m_j},\quad \boldsymbol b^{\bsnu}=\prod_{j\in\supp{(\bsnu)}}b_j^{\nu_j},\quad \partial_{\bsy}^{\bsnu}=\prod_{j\in\supp(\bsnu)}\frac{\partial^{\nu_j}}{\partial y_j^{\nu_j}},
$$
where $\bsy\in U$ and we use the convention $0^0:=1$.

Let $D\subset\mathbb R^d$ be a non-empty Lipschitz domain for $d\in\{1,2,3\}$. We denote the subspace of $H^1(D)$ with vanishing trace on $\partial D$ as $H_0^1(D)$. The dual space of $H_0^1(D)$ is denoted by $H^{-1}(D)$ with respect to the pivot space $L^2(D)$.

\begin{definition}\rm 
Let $D\subset \mathbb R^{d}$ be a non-empty Lipschitz domain for $d\in\{1,2,3\}$. We define
$$
\|v\|_{L^\infty(D)}:=\begin{cases}
\underset{\bsx\in D}{\rm ess\,sup}\,\|v(\bsx)\|_{\ell^2}&\text{if}~v\!:D\to \mathbb R^d,\\
\underset{\bsx\in D}{\rm ess\,sup}\,\|v(\bsx)\|_{2}&\text{if}~v\!:D\to\mathbb R^{d\times d},
\end{cases}
$$
where $\|\cdot\|_{\ell^2}$ denotes the Euclidean norm and $\|\cdot\|_2$ denotes the matrix spectral norm.
In addition, if $v\!:D\to \mathbb R^d$, we define
$$
\|v\|_{W^{1,\infty}(D)}:=\max\big\{\underset{\bsx\in D}{\rm ess\,sup}\,\|v(\bsx)\|_{\ell^2},\underset{\bsx\in D}{\rm ess\,sup}\,\|v'(\bsx)\|_2\big\},
$$
where $v'\!:D\to\mathbb R^{d\times d}$ denotes the Jacobian matrix of $v$. Furthermore, we will make use of the Sobolev norms defined as
\begin{align*}
&\|v\|_{H_0^1(D)}:=\|\nabla v\|_{L^2(D)}=\left(\int_D\|\nabla v(\bsx)\|_{\ell^2}^2\,\text{d} \bsx\right)^\frac{1}{2}, \qquad v \in H_0^1(D),\\
&\|v\|_{H_0^1(D)\cap H^2(D)}:=\left(\|v\|^2_{L^2(D)}+\|\Delta v\|^2_{L^2(D)}\right)^\frac{1}{2}, \quad v \in H_0^1(D)\cap H^2(D).
\end{align*}
Furthermore, when $v\in\mathcal C^k(\overline{D})$ for $k\in\mathbb N_0$, we define
$$
\|v\|_{\mathcal C^k(\overline{D})}:=\max_{|\bsnu|\leq k} \sup_{\bsx\in \overline{D}}|\partial_{\bsx}^{\bsnu}v(\bsx)|.
$$

We omit the domain $D$ from the subscript when there is no risk of confusion.
\end{definition}

 Furthermore, if $V$ is a Sobolev space and $I\subset\mathbb R$ is an interval, we denote by $L^2(I;V)$ the space of all measurable functions $v\!:I\to V$ with finite norm
$$
\|v\|_{L^2(I;V)}:=\bigg(\int_I \|v(\cdot,t)\|_{V}^2\,{\rm d}t\bigg)^{1/2}.
$$

We assume that the reference domain $D_{\rm ref}\subset \mathbb R^d$, $d\in\{1,2,3\}$, is bounded and non-empty with Lipschitz regular boundary, and that the mapping
\begin{align*}
\boldsymbol V\!:\overline{D_{\rm ref}}\times U\to \mathbb R^d
\end{align*}
is a vector field. We define the family of admissible domains $\{D(\bsy)\}_{\bsy\in U}$, parameterized by
$$
D(\bsy):=\boldsymbol V(D_{\rm ref},\bsy),\quad \bsy\in U,
$$
and the \emph{hold-all domain} is defined by setting
\begin{equation}\label{hold_all_domain}
\mathcal D:=\bigcup_{\bsy\in U}D(\bsy).
\end{equation}
Furthermore, we denote by $J(\cdot,\bsy)\!:D_{\rm ref}\to\mathbb R^{d\times d}$ the Jacobian matrix of vector field $\boldsymbol V$ with respect to $\bsx\in D_{\rm ref}$. In addition, we define the dimensionally truncated domain mapping $\boldsymbol 
V_s\!:\overline{D_{\rm ref}}\times U\to\mathbb R^d$ by setting
$$
\boldsymbol V_s(\bsx,\bsy):=\boldsymbol V(\bsx,(y_1,\ldots,y_s,0,0,\ldots)),\quad \bsx\in \overline{D_{\rm ref}}\times U.
$$

Instead of following the approaches taken by~\cite{HPS16}~or~\cite{HHKKS24}, we carry out the analysis using the arguments of~\cite{Chernov23a,Chernov23b}. As noted by \citet[Remark 1]{HarbrechtSchmidlinMultilevel}, instead of assuming a particular kind of \emph{parametric representation} for $\boldsymbol V(\bsx,\bsy)$, such as a Karhunen--Lo\`eve expansion, we can instead only impose a condition on the growth rate of the parametric derivatives of $\boldsymbol V(\bsx,\bsy)$.  We assume that the random domain mapping $\boldsymbol V$ is parameterized by a countable sequence $\bsy=(y_j)_{j\geq 1}$ of independent and identically distributed (i.i.d.) uniform random variables $y_j=y_j(\omega)$, $j\geq 1$, supported on $[-\frac12,\frac12]$. The main modeling assumptions regarding the random domain mapping and other problem parameters are stated below. Crucially for our analysis, the parametric regularity bound for the input random field is quantified in Assumption~\ref{assump:gevrey}.%

\begin{assumption} $ $
\begin{list}{{\rm (A\arabic{assumpenum})}}{\usecounter{assumpenum}%
  \setlength{\leftmargin}{1.7\parindent}%
  \setlength{\labelwidth}{1.7\parindent}%
  \setlength{\labelsep}{0.5em}%
  \renewcommand{\makelabel}[1]{\hfill #1}}
\item We assume that for each $\bsy\in U$, $\boldsymbol V(\cdot,\bsy)\!:\overline{D_{\rm ref}}\to\mathbb R^d$ is an invertible and twice continuously differentiable vector field, with\label{a1}
$$
\sup_{\bsy\in U}\|\boldsymbol V(\cdot,\bsy)\|_{\mathcal C^2(\overline{D_{\rm ref}})}<\infty\quad\text{and}\quad \sup_{\bsy\in U}\|\boldsymbol V^{-1}(\cdot,\bsy)\|_{\mathcal C^2(\overline{D(\bsy)})}<\infty.
$$
\item We assume that there exist uniform bounds $0<\sigma_{\min}\leq 1\leq \sigma_{\max}<\infty$ such that  $\text{for all}~\bsx\in D_{\rm ref},~\bsy\in U$\label{a2sv}
$$
\sigma_{\min}\leq \min\sigma(J(\bsx,\bsy))\leq \max\sigma(J(\bsx,\bsy))\leq\sigma_{\max},
$$
holds, where $\sigma(M)$ denotes the set of all singular values of matrix $M$.
\item {\bf\emph{Gevrey regularity of the domain mapping:}} We assume that the vector field $\boldsymbol V(\cdot,\bsy)$ is infinitely many times continuously differentiable with respect to $\bsy\in U$  and there exists a constant $C\geq 1$, an exponent $\beta\geq 1$, and a sequence $\boldsymbol b=(b_j)_{j\geq 1}$ of non-negative numbers such that\label{assump:gevrey}
$$
\|\partial_\bsy^{\bsnu}\boldsymbol V(\cdot,\bsy)\|_{W^{1,\infty}(D_{\rm ref})}\leq C(|\bsnu|!)^\beta\boldsymbol b^{\bsnu}\quad\text{for all}~\bsnu\in\mathscr F,~\bsy\in U.
$$
It is an immediate consequence of the above condition that
$$
\|\partial_{\bsy}^{\bsnu}J(\cdot,\bsy)\|_{L^\infty(D_{\rm ref})}\leq C(|\bsnu|!)^\beta\boldsymbol b^{\bsnu}\quad\text{for all}~\bsnu\in\mathscr F,~\bsy\in U.
$$
Furthermore, it follows by continuity that
$$
\lim_{s\to\infty}\boldsymbol V(\cdot,\bsy_{\leq s})=\boldsymbol V(\cdot,\lim_{s\to\infty}\bsy_{\leq s})=\boldsymbol V(\cdot,\bsy)
$$
and hence, for all $\bsy \in U$ there holds
$$
\lim_{s\to\infty}\|\boldsymbol V(\cdot,\bsy)-\boldsymbol V_s(\cdot,\bsy)\|_{W^{1,\infty}(D_{\rm ref})}=0.
$$
\item {\bf\emph{Gevrey regularity of the source term:}} The source term $f\!:\mathcal D\to\mathbb R$ is infinitely many times continuously differentiable and satisfies\label{a4}
$$
\|\partial_{\bsx}^{\bsnu}f\|_{L^\infty(\mathcal D)}\leq C_f (|\bsnu|!)^\beta\boldsymbol \rho^{\bsnu}\quad\text{for all}~\bsnu\in\mathscr F,
$$
where $C_f>0$ is a constant depending on $f$ and $\boldsymbol\rho=(\rho_j)_{j=1}^d$ is a sequence of non-negative real numbers. Without loss of generality, we assume that the Gevrey exponent $\beta$ coincides with that of Assumption~\ref{assump:gevrey}.
\item The sequence $\boldsymbol b$ in  Assumption~\ref{assump:gevrey} satisfies $\boldsymbol b\in \ell^p(\mathbb N)$ for some $p\in(0,1)$.\label{a6} 
\item  The sequence $\boldsymbol b$ in  Assumption~\ref{assump:gevrey} satisfies 
$$
b_1\geq b_2\geq b_3\geq \cdots.
$$
\item The reference domain $D_{\rm ref}\subset\mathbb R^d$ is a convex polyhedron.\label{a9}
\end{list}
\end{assumption}

We remark that assumption \ref{a9} can be relaxed if one allows for curvilinear triangulations or non domain-conformity. However, for simplicity we only use finite elements on domain conforming non-curvilinear triangulations.

In this article we shall consider two model problems subject to domain uncertainty.
\paragraph{Poisson equation.} The  Poisson equation over a parametric domain $D(\bsy)$ is the problem of finding, for all $\bsy\in U$, $u(\cdot,\bsy)\!:D(\bsy)\to\mathbb R$ such that
$$
\begin{cases}
-\Delta u(\bsx,\bsy)=f(\bsx)&\text{for}~\bsx\in D(\bsy),\\
u(\bsx,\bsy)=0&\text{for}~\bsx\in \partial D(\bsy),
\end{cases}%
$$
where $f\!:\mathcal D\to\mathbb R$ is a given source term over the hold-all domain $\mathcal D$, see \eqref{hold_all_domain}. The weak formulation over $D(\bsy)$ can be stated as: for each $\bsy\in U$, find $u(\cdot,\bsy)\in H_0^1(D(\bsy))$ such that
$$
\int_{D(\bsy)}\nabla  u(\bsx,\bsy)\cdot \nabla  v(\bsx)\,{\rm d}\bsx=\int_{D(\bsy)}f(\bsx)v(\bsx)\,{\rm d}\bsx\quad\text{for all}~v\in H_0^1(D(\bsy)).
$$
We consider the path-wise formulation since this allows for straightforward analysis of the higher-order partial derivatives of the PDE solution with respect to the parameter $\bsy$.
However, for the purposes of analysis, it is convenient to analyze instead the pullback solution $\widehat u(\bsx,\bsy)=u(\boldsymbol V(\bsx,\bsy),\bsy)$. This is equivalent to the following pullback weak formulation obtained using a change of variables: for all $\bsy\in U$, find  $\widehat u(\cdot,\bsy)\in H_0^1(D_{\rm ref})$ such that
\begin{align}\label{eq:weakform}
\int_{D_{\rm ref}}(A(\bsx,\bsy)\nabla \widehat u(\bsx,\bsy))\cdot \nabla \widehat v(\bsx)\,{\rm d}\bsx=\int_{D_{\rm ref}}f_{\rm ref}(\bsx,\bsy)\widehat v(\bsx)\,{\rm d}\bsx
\end{align}
for all$~\widehat v\in H_0^1(D_{\rm ref})$,  where 
\begin{align*}
  f_{\rm ref}(\bsx,\bsy)&=f(\boldsymbol V(\bsx,\bsy))\det J(\bsx,\bsy) \quad\mathrm{and}\\
  A(\bsx,\bsy)&=(J(\bsx,\bsy)^{\rm T}J(\bsx,\bsy))^{-1}\det J(\bsx,\bsy).
\end{align*}
Note that the test functions $\widehat v$ are independent of the parameter $\bsy\in U$ due to the fact that the spaces $H_0^1(D_{\rm{ref}})$ and $H_0^1(D(\bsy))$ are isomorphic as shown by \citet[Lemma 1]{HPS16}.
Since the parameters $\bsy\in U$ are uniformly bounded, it is a consequence of the Lax--Milgram lemma that $\widehat u$ can be bounded independently of $\bsy$. In particular, $\widehat u\in L^p(U;H_0^1(D_{\rm ref}))$ for all $1\leq p\leq \infty$.

\paragraph{Heat equation.} The  heat equation over a parametric domain $D(\bsy)$ and the time interval $I=[0,T]$ for some fixed terminal time $T>0$  is the problem of finding, for all $\bsy\in U$, $u(\cdot,\cdot,\bsy)\!:D(\bsy)\times I\to\mathbb R$ such that
$$
\begin{cases}
\frac{\partial}{\partial t}u(\bsx,t,\bsy)-\Delta u(\bsx,t,\bsy)=f(\bsx)&\text{for}~\bsx\in D(\bsy),~t\in (0,T]\\
u(\bsx,0,\bsy)=u_0(\bsx),&\text{for}~\bsx\in D(\bsy),\\
u(\bsx,t,\bsy)=0&\text{for}~\bsx\in \partial D(\bsy),~t\in (0,T],
\end{cases}%
$$
where $f\!:\mathcal D\to\mathbb R$ is a given source term and $u_0\!:\mathcal D\to\mathbb R$ a given initial condition. We shall consider the space-time weak formulation of this model problem (cf., e.g., \cite{kunothschwab,schwabstevenson}), since this is independent of the chosen PDE discretization scheme, thereby accommodating a very general analysis of the parametric regularity of the PDE solution. Moreover, it allows us to conveniently enforce the initial condition almost everywhere as part of the variational formulation. The space-time weak formulation over $D(\bsy)\times I$ can be stated as: %
\begin{align*}
&\int_I \int_{D(\bsy)}\tfrac{\partial}{\partial t}u(\bsx,t,\bsy) v_1(\bsx,t) \,{\rm d}\bsx\,{\rm d}t+\int_I \int_{D(\bsy)} \nabla u(\bsx,t,\bsy)\cdot \nabla v_1(\bsx,t)\,{\rm d}\bsx\,{\rm d}t\\
&\quad +\int_{D(\bsy)} u(\bsx,0,\bsy)v_2(\bsx)\,{\rm d}\bsx\\
&=\int_I \int_{D(\bsy)}f(\bsx)v_1(\bsx,t)\,{\rm d}\bsx\,{\rm d}t + \int_I \int_{D(\bsy)}u_0(\bsx)v_2(\bsx)\,{\rm d}\bsx
\end{align*}
for all $(v_1,v_2)\in L^2(I;H_0^1(D(\bsy)))\times L^2(D(\bsy))$. 
However, for the purposes of analysis, it is convenient to analyze instead the pullback solution given by
\begin{align*}
  \widehat u(\bsx,t,\bsy)=u(\boldsymbol V(\bsx,\bsy),t,\bsy).  
\end{align*}
By change of variables, the pullback space-time variational formulation is given by
\begin{align}\label{spacetimepullback}
b(\bsy;\widehat u,\widehat v)=F(\bsy;\widehat v)\quad\text{for all}~\widehat v=(\widehat v_1,\widehat v_2)\in\mathcal Y~\text{and}~\bsy\in U,
\end{align}
where 
\begin{align*}
b(\bsy;\widehat u,\widehat v):=&\int_I \int_{D_{\rm ref}} \tfrac{\partial}{\partial t}\widehat u(\bsx,t,\bsy)\widehat v_1(\bsx,t)\det J(\cdot,\bsy)\,{\rm d}\bsx\,{\rm d}t\\
&\quad +\int_I\int_{{D_{\rm ref}}} A(\bsx,\bsy)\nabla \widehat u(\bsx,t,\bsy)\cdot \nabla \widehat v_1(\bsx,t)\,{\rm d}\bsx\,{\rm d}t\\
&\quad+\int_{D_{\rm ref}} \widehat u(\bsx,0,\bsy)\widehat v_2(\bsx)\det J(\bsx,\bsy)\,{\rm d}\bsx
\end{align*}
and
\begin{align*}
F(\bsy;\widehat v):=&\int_I \int_{D_{\rm ref}}  f_{\rm ref}(\bsx,\bsy)\widehat v_1(\bsx,t)\,{\rm d}\bsx\,{\rm d}t\\
&\quad+\int_{D_{\rm ref}} \widehat u_0(\bsx,\bsy)\widehat v_2(\bsx)\det J(\bsx,\bsy)\,{\rm d}\bsx
\end{align*}
and we denote 
\begin{align*}
 f_{\rm ref}(\bsx,\bsy)&=f(\boldsymbol V(\bsx,\bsy))\det J(\bsx,\bsy),\\   
 \widehat u_0(\bsx,\bsy)&=u_0(\boldsymbol V(\bsx,\bsy),\bsy),\\
 A(\bsx,\bsy)&=(J(\bsx,\bsy)^{\rm T}J(\bsx,\bsy))^{-1}\det J(\bsx,\bsy)\quad \mathrm{and}\\
 \mathcal Y\!&:=\!L^2(I;H_0^1(D_{\rm ref}))\times L^2(D_{\rm ref}).
\end{align*}
 We remark that $\widehat u_0(\cdot,\bsy)$ agrees with $\widehat u(\cdot,0,\bsy)$ almost everywhere in $D_{\rm ref}$.%

Moreover, we define the space
\begin{align}\label{eq:x}
\mathcal X:=\{v\in L^2(I;H_0^1(D_{\rm ref}))\mid \tfrac{\partial}{\partial t}v\in L^2(I;H^{-1}(D_{\rm ref}))\}
\end{align}
endowed with the graph norm
$$
\|v\|_{\mathcal X}:=\bigg(\int_I\big(\|v(\cdot,t)\|_{H_0^1(D_{\rm ref})}^2+\|\tfrac{\partial}{\partial t}v(\cdot,t)\|_{H^{-1}(D_\mathrm{ref})}^2\big)\,{\rm d}t\bigg)^{1/2}.
$$
Furthermore, we denote $\mathcal L^2:=L^2(D_{\rm ref}\times I)$.%

For our analysis of the pullback heat equation, we need to impose an additional conditions on the regularity of the initial  condition  $u_0$ as well as on the domain mapping.  Hence, we state the following assumptions exclusively for this setting:
\begin{assumption}$ $
\begin{list}{{\rm (A\arabic{assumpenum})}}{\usecounter{assumpenum}%
  \setlength{\leftmargin}{2\parindent}%
  \setlength{\labelwidth}{2\parindent}%
  \setlength{\labelsep}{0.5em}%
  \renewcommand{\makelabel}[1]{\hfill #1}}
\setcounter{assumpenum}{7}
\item  The initial condition $u_0\!:\mathcal D\to\mathbb R$ is infinitely many times continuously differentiable and satisfies\label{a10}
$$
\|\partial_{\bsx}^{\bsnu}u_0\|_{L^2(\mathcal D)}\leq C_{u_0}\quad\text{for all}~\bsnu\in \mathbb N_0^d
$$
for some constant $C_{u_0}>0$.
\item There exist a constant $C\geq 1$, an exponent $\beta\geq 1$, and a sequence $\bsb=(b_j)_{j\geq 1}$ of non-negative numbers such that
\begin{align*}
    \|\partial_\bsy^\bsnu\boldsymbol{V}(\cdot,\bsy)\|_{W^{2,\infty}(D_\mathrm{ref})}\leq C(|\bsnu|!)^\beta \bsb^\bsnu\quad\text{for all}\quad \bsnu\in\mathscr{F},\bsy\in U.
\end{align*}
From this condition it follows directly that
\begin{align*}
    \|\partial_\bsy^\bsnu J(\cdot,\bsy)\|_{W^{1,\infty}(D_\mathrm{ref})}\leq C(|\bsnu|!)^\beta \bsb^\bsnu\quad\text{for all}\quad \bsnu\in\mathscr{F},\bsy\in U.
\end{align*}\label{A10}
\item There holds $\partial_t\widehat u(\cdot,\cdot,\bsy)\in \mathcal L^2$ for all $\bsy \in U$.\label{timederivL2}
\end{list}
\end{assumption}

We note that the regularity analysis for the parabolic problem can also be
carried out with the time derivative belonging to
\(L^2(I;H^{-1}(D_{\mathrm{ref}}))\), since this still yields the regularity
needed for the parametric derivatives of the solution (cf. Lemma \ref{L2reg}). However, for clarity of
presentation, we carry out the analysis under Assumption~\ref{timederivL2}.

We will show that under these assumptions, the pullback $\widehat u_0=u_0\circ \boldsymbol V(\cdot,\bsy)$ will belong to a Gevrey class and that for the parametric derivatives of the pullback solution there holds $\partial_\bsy^\bsnu\widehat u\in L^2(I;H_0^1(D_\mathrm{ref}))$ for all $\bsnu\in\mathscr F$, allowing us to complete the parametric regularity analysis for the parabolic model problem.

In what follows, we study QMC approximations of the expected value
$$
\mathbb E[\widehat u]=\int_U \widehat u(\bsy)\,{\rm d}\bsy,
$$
where $\widehat u(\cdot,\bsy)\in H_0^1(D_{\rm ref})$ if $\widehat u$ is the solution of the pullback Poisson equation~\eqref{eq:weakform} and $\widehat u(\cdot,\cdot,\bsy)\in L^2(I;H_0^1(D_{\rm ref}))$ if $\widehat u$ is the solution of the pullback heat equation~\eqref{spacetimepullback}.

\section{Quasi-Monte Carlo integration}\label{sec:qmc}

Let $U_s:=[-\frac12,\frac12]^s$. Consider the task of integrating a continuous function $F\!:U_s\to \mathbb R$ over an $s$-dimensional domain $U_s$
\[
I_s(F):=\int_{U_s}F(\boldsymbol{y})\,{\rm d}\boldsymbol{y}.
\]
A randomly shifted lattice rule (cf., e.g., \cite{dick2013high,dick2022lattice,kuonuyenssurvey,KuoNuyens2018}) is a method for QMC integration, represented as:
\begin{align}\label{QMC_rule}
&Q_{\text{ran}}(F):=\frac{1}{R}\sum_{r=1}^R Q_r(F),\quad\text{with}~Q_r(F):=\frac{1}{n}\sum_{i=1}^n F(\{\bst_i+\boldsymbol\Delta^{(r)}\}-\tfrac{\mathbf 1}{\mathbf 2}),
\end{align}
where $\boldsymbol\Delta^{(1)},\ldots,\boldsymbol\Delta^{(R)}$ are i.i.d.~random shifts drawn from $\mathcal U([0,1]^s)$, $\{\cdot\}$ denotes the componentwise fractional part, and the lattice points are given by
\[
\boldsymbol t_i:=\bigg\{\frac{i\boldsymbol z}{n}\bigg\},\quad i\in\{1,\ldots,n\},
\]
where $\boldsymbol z\in\{0,\ldots,n-1\}^s$ denotes the generating vector.
Assume that the function $F$ belongs to a weighted unanchored Sobolev space (cf., e.g., \cite{dick2013high,dick2022lattice,kuonuyenssurvey,KuoNuyens2018}) of dominating first-order smoothness $\mathcal W_{{s,\boldsymbol\gamma}}={\rm cl}_{\|\cdot\|_{\mathcal W_{s,\bsgamma}}}(\mathcal C^\infty(U_s))$ equipped with the norm
\[
\|F\|_{{\mathcal W_{s,\boldsymbol\gamma}}}:=\left(\sum_{\mathfrak u\subset\{1,\ldots,s\}}\frac{1}{\gamma_{\mathfrak u}}\int_{U_{|\mathfrak u|}}\left(\int_{U_{s-|\mathfrak u|}}\frac{\partial^{|\mathfrak u|}}{\partial \boldsymbol{y}_{\mathfrak u}}F(\boldsymbol{y})\,{\rm d}\boldsymbol{y}_{-\mathfrak u}\right)^2\,{\rm d}\boldsymbol{y}_{\mathfrak u}\right)^{1/2},
\]
where $\boldsymbol\gamma:=(\gamma_{\mathfrak u})_{\mathfrak u\subset\{1,\ldots,s\}}$ is a collection of positive weights, $\bsy_{\mathfrak u}:=(y_j)_{j\in\setu}$, and $\bsy_{-\mathfrak u}:=(y_j)_{j\in\{1,\ldots,s\}\setminus \mathfrak u}$. Then there exists a sequence of generating vectors using a component-by-component (CBC) algorithm~\citep{ckn06,dick2013high,cn06} with the following error bound on the integration error.
\begin{lemma}[cf.~{\citet[Theorem~5.1]{kuonuyenssurvey}}]\label{lemma:affineqmc}
Let $F\in\mathcal W_{s,\bsgamma}$ with positive weights $\boldsymbol{\gamma}=(\gamma_{\mathfrak u})_{\mathfrak u\subset\{1,\ldots,s\}}$. An $s$-dimensional randomly shifted lattice rule with $n=2^m$, $m\in\mathbb N$, points can be constructed by a CBC algorithm such that, for $R$ independent random shifts and for all $\lambda\in (1/2,1]$:
\[
\sqrt{\mathbb{E}_{\boldsymbol{\Delta}}\left[|I_{s}(F)-Q_{\rm ran}(F)|^2\right]}\leq\frac{1}{\sqrt{R}}\left(\frac{2}{n}\!\!\sum_{\varnothing\neq \mathfrak u\subset\{1,\ldots,s\}}\!\!\gamma_{\mathfrak u}^\lambda\bigg(\frac{2\zeta(2\lambda)}{(2\pi^2)^\lambda}\bigg)^{|\mathfrak u|}\right)^{\frac{1}{2\lambda}}\!\!\|F\|_{\mathcal W_{s,\boldsymbol\gamma}},
\]
where $\mathbb E_{\boldsymbol\Delta}[\cdot]$ denotes the expected value with respect to uniformly distributed random shifts over $[0,1]^s$ and $\zeta(x):=\sum_{k=1}^\infty k^{-x}$ is the Riemann zeta function for $x>1$.
\end{lemma}

\section{Parametric regularity analysis}\label{sec:4}
Our first task  is to analyze $\partial_{\bsy}^{\bsnu}A(\bsx,\bsy)$ since it appears in the variational formulations for both the stationary and time-dependent model problems. We do this step-by-step:
\begin{addmargin}[3em]{0em}
\begin{enumerate}
\item[\em Step 1:] We first derive a regularity estimate for $\partial_{\bsy}^{\bsnu}(J(\bsx,\bsy)^{\rm T}J(\bsx,\bsy))$.%
\item[\em Step 2:] Step 1 will be used to deduce a bound for $\partial_{\bsy}^{\bsnu}(J(\bsx,\bsy)^{\rm T}J(\bsx,\bsy))^{-1}$.%
\item[\em Step 3:] We derive a regularity bound for $\partial_{\bsy}^{\bsnu}\det J(\bsx,\bsy)$.%
\item[\em Step 4:] Finally, we use the Leibniz product rule to combine steps 2 and 3 to obtain a bound for $\partial_{\bsy}^{\bsnu}((J(\bsx,\bsy)^{\rm T}J(\bsx,\bsy))^{-1}\det J(\bsx,\bsy))$.%
\end{enumerate}
\end{addmargin}
This analysis is carried out in Subsection~\ref{sec:Aregularity}

We shall also be interested in the parametric regularity of the pullback source term $f_{\rm ref}$. Our approach is as follows:
\begin{addmargin}[3em]{0em}
\begin{enumerate}
\item[\em Step 5:] We use Fa\`a di Bruno's formula to obtain a regularity bound for $f(\boldsymbol V(\bsx,\bsy))$.
\item[\em Step 6:] The Leibniz product rule is used in conjunction with step 5 to obtain a bound for $\partial_{\bsy}^{\bsnu}f_{\rm ref}(\bsx,\bsy)$.
\end{enumerate}
\end{addmargin}
Parametric regularity analysis for the right-hand side will be carried out in Subsection~\ref{sec:rhsreg}.

\subsection{Parametric regularity of the pullback coefficient}\label{sec:Aregularity}

\paragraph{Step 1.}We start by proving the following result.
\begin{lemma}\label{step1lemma} Under Assumptions~\ref{a1} and~\ref{assump:gevrey}, there holds
$$
\|\partial_{\bsy}^{\bsnu}(J(\cdot,\bsy)^{\rm T}J(\cdot,\bsy))\|_{L^\infty(D_{\rm ref})}\leq C^2((|\bsnu|+1)!)^{\beta}\boldsymbol b^{\bsnu}
$$
for all $\bsy\in U$ and $\bsnu\in\mathscr F$.
\end{lemma}\begin{proof}By the Leibniz product rule and Assumption~\ref{assump:gevrey}, we obtain for any $\bsnu\in\mathscr F$ that
\begin{align*}
\|\partial_{\bsy}^{\bsnu}(J(\cdot,\bsy)^{\rm T}J(\cdot,\bsy))\|_{L^\infty}&=\bigg\|\sum_{\boldsymbol m\leq \bsnu}\binom{\bsnu}{\boldsymbol m}\partial_{\bsy}^{\boldsymbol m}J(\cdot,\bsy)^{\rm T}\partial_{\bsy}^{\bsnu-\boldsymbol m}J(\cdot,\bsy)\bigg\|_{L^\infty}\\
&\leq C^2\boldsymbol b^{\bsnu}\sum_{\boldsymbol m\leq \bsnu}\binom{\bsnu}{\boldsymbol m}(|\boldsymbol m|!)^\beta(|\bsnu-\boldsymbol m|!)^\beta\\
&=C^2\boldsymbol b^{\bsnu}\sum_{\ell=0}^{|\bsnu|}(\ell!)^\beta((|\bsnu|-\ell)!)^\beta\sum_{\substack{\boldsymbol m\leq \bsnu\\|\boldsymbol m|=\ell}}\binom{\bsnu}{\boldsymbol m}\\
&=C^2\boldsymbol b^{\bsnu}\sum_{\ell=0}^{|\bsnu|}(\ell!)^\beta((|\bsnu|-\ell)!)^\beta\binom{|\bsnu|}{\ell}\\
&= C^2 |\bsnu|!\boldsymbol b^{\bsnu}\sum_{\ell=0}^{|\bsnu|}(\ell!)^{\beta-1}((|\bsnu|-\ell)!)^{\beta-1}\\
&\leq C^2 (|\bsnu|!)^\beta \boldsymbol b^{\bsnu}(|\bsnu|+1)\\
&\leq C^2((|\bsnu|+1)!)^\beta \boldsymbol b^{\bsnu},
\end{align*}
where we used the Vandermonde convolution $\sum_{\boldsymbol m\leq \bsnu,~|\boldsymbol m|=\ell}\binom{\bsnu}{\boldsymbol m}=\binom{|\bsnu|}{\ell}$ (cf., e.g., \citet[Equation~(5.1)]{gouldbook}) and the inequality $\ell!(v-\ell)!\leq v!$ which holds for all integers $0\leq \ell\leq\nu$.\end{proof}

\paragraph{Step 2.} Let $B(\bsx,\bsy):=J(\bsx,\bsy)^{\rm T}J(\bsx,\bsy)$. Since
$$
B(\bsx,\bsy)B(\bsx,\bsy)^{-1}=I_{d\times d},
$$
we can use the Leibniz product rule to obtain, for any $\bsnu\in\mathscr F\setminus\{\mathbf 0\}$, that
\begin{align*}
\sum_{\boldsymbol m\leq \bsnu}\binom{\bsnu}{\boldsymbol m}\partial_{\bsy}^{\boldsymbol m}B(\bsx,\bsy)\partial_{\bsy}^{\bsnu-\boldsymbol m}B(\bsx,\bsy)^{-1}=0.
\end{align*}
Separating out the $\boldsymbol m=\mathbf 0$ term, we obtain
$$
\partial_{\bsy}^{\bsnu}B(\bsx,\bsy)^{-1}=-B(\bsx,\bsy)^{-1}\sum_{\mathbf 0\neq \boldsymbol m\leq \bsnu}\binom{\bsnu}{\boldsymbol m}\partial_{\bsy}^{\boldsymbol m}B(\bsx,\bsy)\partial_{\bsy}^{\bsnu-\boldsymbol m}B(\bsx,\bsy)^{-1}.
$$
By Assumptions~\ref{a1}\,--\,\ref{assump:gevrey}, there holds $\|B(\cdot,\bsy)^{-1}\|_{L^\infty}\leq \sigma_{\min}^{-2}$ and we obtain
\begin{align*}
\|\partial_\bsy^{\bsnu}B(\cdot,\bsy)^{-1}\|_{L^\infty}\leq C^2\sigma_{\min}^{-2}\!\sum_{\mathbf 0\neq \boldsymbol m\leq \bsnu}\!\binom{\bsnu}{\boldsymbol m}((|\boldsymbol m|+1)!)^\beta\boldsymbol b^{\boldsymbol m}\|\partial_{\bsy}^{\bsnu-\boldsymbol m}B(\cdot,\bsy)^{-1}\|_{L^\infty}.
\end{align*}
Note that we have now expressed the $\boldsymbol\nu^{\rm th}$ order partial derivative in terms of the lower order partial derivatives. Defining $\Lambda_{\bsnu}:=\|\partial_{\bsy}^{\bsnu}B(\cdot,\bsy)^{-1}\|_{L^\infty}$ yields the multidimensional recurrence relation 
\begin{align}
&\Lambda_{\mathbf 0}\leq \sigma_{\min}^{-2},\label{eq:multirecu1}\\
&\Lambda_{\bsnu}\leq C^2\sigma_{\min}^{-2}\sum_{\mathbf 0\neq \boldsymbol m\leq \bsnu}\binom{\bsnu}{\boldsymbol m}((|\boldsymbol m|+1)!)^\beta\boldsymbol b^{\boldsymbol m}\Lambda_{\bsnu-\boldsymbol m}\quad\text{for all}~\bsnu\in\mathscr F\setminus\{\mathbf 0\}.\label{eq:multirecu2}
\end{align}
This recursion admits the following explicit bound.
\begin{lemma}
Let $(\Lambda_{\bsnu})_{\bsnu\in\mathscr F}$ be a sequence satisfying~\eqref{eq:multirecu1}\,--\,\eqref{eq:multirecu2}. Then there holds
$$
\Lambda_{\bsnu}\leq \sigma_{\min}^{-2|\bsnu|-2}C^{2|\bsnu|}a_{\bsnu}\boldsymbol b^{\bsnu}\quad\text{for all}~\bsnu\in \mathscr F,
$$
where
\begin{align*}
a_{\mathbf 0}=1\quad\text{and}\quad a_{\bsnu}=\sum_{\substack{\boldsymbol m\leq \bsnu\\ \boldsymbol m\neq\mathbf 0}}\binom{\bsnu}{\boldsymbol m}((|\boldsymbol m|+1)!)^\beta a_{\bsnu-\boldsymbol m}\quad\text{for all}~\bsnu\in\mathscr F\setminus\{\mathbf 0\}.
\end{align*}
\end{lemma}
\begin{proof} We prove the claim by induction with respect to the order of the multi-index $\bsnu$. The basis of the induction $\bsnu=\mathbf 0$ follows immediately. We fix $\bsnu\in\mathscr F\setminus\{\mathbf 0\}$ and assume that the claim holds for all multi-indices with order less than $|\bsnu|$. Then
\begin{align*}
\Lambda_{\bsnu}&\leq C^2\sigma_{\min}^{-2}\sum_{\mathbf 0\neq \boldsymbol m\leq \bsnu}\binom{\bsnu}{\boldsymbol m}((|\boldsymbol m|+1)!)^\beta\boldsymbol b^{\boldsymbol m}\Lambda_{\bsnu-\boldsymbol m}\\
&\leq C^2\sigma_{\min}^{-2}\sum_{\mathbf 0\neq \boldsymbol m\leq \bsnu}\binom{\bsnu}{\boldsymbol m}((|\boldsymbol m|+1)!)^\beta\boldsymbol b^{\boldsymbol m}\sigma_{\min}^{-2|\bsnu|+2|\boldsymbol m|-2}\\
&\quad\quad\quad\quad\quad\quad\times C^{2|\bsnu|-2|\boldsymbol m|}a_{\boldsymbol \nu-\boldsymbol m}\boldsymbol b^{\boldsymbol\nu-\boldsymbol m}\\
&\leq C^{2|\bsnu|}\sigma_{\min}^{-2|\bsnu|-2}\boldsymbol b^{\bsnu}\sum_{\mathbf 0\neq \boldsymbol m\leq \bsnu}\binom{\bsnu}{\boldsymbol m}((|\boldsymbol m|+1)!)^\beta a_{\boldsymbol \nu-\boldsymbol m}\\
&=C^{2|\bsnu|}\sigma_{\min}^{-2|\bsnu|-2}\boldsymbol b^{\bsnu}a_{\bsnu},
\end{align*}
as desired.\end{proof}

By Lemmata~\ref{lemma:taubnd}--\ref{lemma:taubound}, we conclude the following.
\begin{lemma} Under Assumptions~\ref{a1}\,--\,\ref{assump:gevrey}, there holds
$$
\|\partial_\bsy^{\bsnu}B(\cdot,\bsy)^{-1}\|_{L^\infty(D_{\rm ref})}\leq \sigma_{\min}^{-2|\bsnu|-2}C^{2|\bsnu|}(|\bsnu|!)^\beta 4^{\beta |\bsnu|}\max\{1,2^{|\bsnu|-1}\}\boldsymbol b^{\bsnu}
$$
for all $\bsnu\in \mathscr F$ and $\bsy\in U$.
\end{lemma}

\paragraph{Step 3.} We begin by proving an auxiliary result needed for the regularity bound of $\partial_{\bsy}^{\bsnu}\det J(\bsx,\bsy)$. 
\begin{lemma}\label{lemma:Jinvbound}
 Under Assumptions~\ref{a1}\,--\,\ref{assump:gevrey}, there holds
 \begin{align*}
    \|\partial_\bsy^\bsnu J(\cdot,\bsy)^{-1}\|_{L^\infty(D_{\rm ref})}\leq C^{|\bsnu|}\sigma_{\min}^{-|\bsnu|-1}\max\{1,2^{|\bsnu|-1}\}(|\bsnu|!)^\beta\boldsymbol{b}^\bsnu
\end{align*}
for all $\bsy\in U$ and $\bsnu\in\mathscr F$.
\end{lemma}
\begin{proof}
Let $\bsnu\in\mathscr F\setminus\{\mathbf 0\}$. Since \begin{align*}
    J(\bsx,\bsy)J(\bsx,\bsy)^{-1}=I_{d\times d},
\end{align*}
we obtain by the Leibniz product rule that
\begin{align*}
    \sum_{\boldsymbol{m}\leq \bsnu}\binom{\bsnu}{\boldsymbol m}\partial_\bsy^{\boldsymbol{m}}J(\bsx,\bsy)\partial_\bsy^{\bsnu-\boldsymbol{m}}J(\bsx,\bsy)^{-1}=0,
\end{align*}
which is equivalent to
\begin{align*}
\partial_\bsy^\bsnu J(\bsx,\bsy)^{-1}=-J(\bsx,\bsy)^{-1}\sum_{\boldsymbol{0}\neq\boldsymbol{m} \leq \bsnu}\binom{\bsnu}{\boldsymbol{m}}\partial_\bsy^{\boldsymbol{m}}J(\bsx,\bsy)\partial_\bsy^{\bsnu-\boldsymbol{m}}J(\bsx,\bsy)^{-1}.
\end{align*}
It follows that 
\begin{align*}
&\|\partial_\bsy^\bsnu J(\cdot,\bsy)^{-1}\|_{L^\infty}\\
\leq &\|J(\cdot,\bsy)^{-1}\|_{L^\infty}\sum_{\boldsymbol{0}\neq\boldsymbol{m}\leq \bsnu}\binom{\bsnu}{\boldsymbol{m}}\|\partial_\bsy^{\boldsymbol{m}}J(\cdot,\bsy)\|_{L^\infty}\|\partial_\bsy^{\bsnu-\boldsymbol{m}}J(\cdot,\bsy)^{-1}\|_{L^\infty}.
\end{align*}
Now using that $\|J(\cdot,\bsy)^{-1}\|_{L^\infty}\leq \sigma_{\min}^{-1}$ and the fact that by Assumption~\ref{assump:gevrey} we have $\|\partial_\bsy^{\boldsymbol{m}}J(\cdot,\bsy)\|_{L^\infty}\leq C(|\boldsymbol{m}|!)^\beta\boldsymbol{b}^{\boldsymbol{m}}$, we find
\begin{align*}
 \|\partial_\bsy^\bsnu J(\cdot,\bsy)^{-1}\|_{L^\infty}\leq C\sigma^{-1}_{\min} \sum_{\boldsymbol{0}\neq\boldsymbol{m}\leq \bsnu}\binom{\bsnu}{\boldsymbol{m}}(|\boldsymbol{m}|!)^\beta\boldsymbol{b}^{\boldsymbol{m}} \|\partial_\bsy^{\bsnu-\boldsymbol{m}}J(\cdot,\bsy)^{-1}\|_{L^\infty}. 
\end{align*}
We define $\xi_\bsnu:=\|\partial_\bsy^\bsnu J(\cdot,\bsy)^{-1}\|_{L^\infty}$. Solving for $\|\partial_\bsy^\bsnu J(\cdot,\bsy)^{-1}\|_{L^\infty}$ in the regularity bound is equivalent to solving 
\begin{align}
    \xi_{\boldsymbol{0}}&\leq \sigma^{-1}_{\min}\label{eq:xirecu1}\\
    \xi_\bsnu &\leq C\sigma^{-1}_{\min} \sum_{\boldsymbol{0}\neq\boldsymbol{m}\leq \bsnu}\binom{\bsnu}{\boldsymbol{m}}(|\boldsymbol{m}|!)^\beta\boldsymbol{b}^{\boldsymbol{m}} \xi_{\bsnu-\boldsymbol{m}}. \label{eq:xirecu2}
\end{align}
for all $\bsnu\in \mathscr{F}\setminus \{\boldsymbol{0}\}$. The claim follows by an application of Lemma~\ref{lemma:xi} from the Appendix.\end{proof}

\begin{lemma} Under Assumptions~\ref{a1}\,--\,\ref{assump:gevrey}, there holds
$$
\|\partial_{\bsy}^{\bsnu}\det J(\cdot,\bsy)\|_{L^\infty(D_{\rm ref})}\leq \sigma_{\max}^d C_{\det J}^{|\bsnu|}\frac{((|\bsnu|+d^2-1)!)^\beta}{((d^2-1)!)^\beta}\boldsymbol b^{\bsnu},%
$$
for all $\bsnu\in\mathscr F,~\bsy\in U$, where 
$$
C_{\det J}=\frac{2^\beta C}{\sigma_{\min}}.
$$
\end{lemma}
\begin{proof} If $\bsnu=\mathbf 0$, then the claim clearly holds since by Assumption \ref{a2sv} and the fact that 
\begin{align*}
    \det J=\prod_{\sigma\in\sigma(J)}\sigma
\end{align*}
we have
\begin{align*}
 \|\det J(\cdot,\bsy)\|_{L^\infty}\leq \sigma_{\max}^d.   
\end{align*}
Next, we let $\bsnu\in\mathscr F$ and $k\geq 1$. We  will proceed by induction and assume that the claim has already been proved for all multi-indices with order $\leq |\bsnu|$. We can use Jacobi's formula (cf., e.g., \citet[Theorem 8.1]{JacobiFormula}) and the Leibniz product rule in conjunction to obtain
\begin{align}\notag
\partial_{\bsy}^{\bsnu+\boldsymbol e_k}\det J(\bsx,\bsy)&=\partial_{\bsy}^{\bsnu}\big(\det J(\bsx,\bsy){\rm tr}\big(J(\bsx,\bsy)^{-1}\partial_{\bsy}^{\boldsymbol e_k}J(\bsx,\bsy))\big)\\
&=\sum_{\boldsymbol m\leq\bsnu}\binom{\bsnu}{\boldsymbol m}\partial_{\bsy}^{\boldsymbol m}\det J(\bsx,\bsy)\partial_{\bsy}^{\bsnu-\boldsymbol m}{\rm tr}\big(J(\bsx,\bsy)^{-1}\partial_{\bsy}^{\boldsymbol e_k}J(\bsx,\bsy)\big),\label{eq:detjind}
\end{align}
where $\boldsymbol e_k\in\mathscr F$ denotes the multi-index whose $k^{\rm th}$ entry is 1 and the other entries are 0. By Lemma~\ref{lemma:Jinvbound}, there holds
$$
\|\partial_{\bsy}^{\bsnu}J(\cdot,\bsy)^{-1}\|_{L^\infty}\leq C^{|\bsnu|}\sigma_{\min}^{-|\bsnu|-1}\max\{1,2^{|\bsnu|-1}\}(|\bsnu|!)^\beta\boldsymbol b^{\bsnu}.
$$
Hence 
\begin{align*}
{\rm tr}\big(J(\bsx,\bsy)^{-1}\partial_{\bsy}^{\boldsymbol e_k}J(\bsx,\bsy)\big)&=\sum_{k=1}^d\sum_{\ell=1}^d [J(\bsx,\bsy)^{-1}]_{k,\ell}[\partial_{\bsy}^{\boldsymbol e_k}J(\bsx,\bsy)]_{\ell,k}
\end{align*}
and an application of the Leibniz product rule reveals that
\begin{align*}
&\|\partial_{\bsy}^{\boldsymbol m}{\rm tr}\big(J(\cdot,\bsy)^{-1}\partial_{\bsy}^{\boldsymbol e_k}J(\cdot,\bsy)\big)\|_{L^\infty}\\
&=\bigg\|\sum_{k=1}^d\sum_{\ell=1}^d \sum_{\boldsymbol w\leq \boldsymbol m}\binom{\boldsymbol m}{\boldsymbol w}[\partial_{\bsy}^{\boldsymbol w}J(\cdot,\bsy)^{-1}]_{k,\ell}[\partial_{\bsy}^{\boldsymbol m-\boldsymbol w+\boldsymbol e_k}J(\cdot,\bsy)]_{\ell,k}\bigg\|_{L^\infty}\\
&\leq C^{|\boldsymbol m|+1}\sigma_{\min}^{-|\boldsymbol m|-1}d^2\boldsymbol b^{\boldsymbol m+\boldsymbol e_k}\\
&\quad \times\sum_{\boldsymbol w\leq \boldsymbol m}\binom{\boldsymbol m}{\boldsymbol w}\max\{1,2^{|\boldsymbol w|-1}\}(|\boldsymbol w|!)^\beta((|\boldsymbol m|-|\boldsymbol w|+1)!)^\beta\\
&\leq C^{|\boldsymbol m|+1}\sigma_{\min}^{-|\boldsymbol m|-1}d^2(|\boldsymbol m|!)^\beta\boldsymbol b^{\boldsymbol m+\boldsymbol e_k}\sum_{\ell=0}^{|\boldsymbol m|}\max\{1,2^{\ell-1}\}(|\boldsymbol m|-\ell+1)^\beta.
\end{align*}
Here, we can estimate
\begin{align*}
\sum_{\ell=0}^{|\boldsymbol m|}\max\{1,2^{\ell-1}\}(|\boldsymbol m|-\ell+1)^\beta&\leq \bigg(\sum_{\ell=0}^{|\boldsymbol m|}\max\{1,2^{\ell-1}\}^{1/\beta}(|\boldsymbol m|-\ell+1)\bigg)^\beta\\
&\leq \bigg(\sum_{\ell=0}^{|\boldsymbol m|}\max\{1,2^{\ell-1}\}(|\boldsymbol m|-\ell+1)\bigg)^\beta\\
&=\big(2^{|\boldsymbol m|+1}-1\big)^\beta\leq 2^{\beta(|\boldsymbol m|+1)},
\end{align*}
where we used the inequality $\sum_{k} a_k\leq \big(\sum_k a_k^\lambda\big)^{1/\lambda}$ which holds for all $a_k\geq 0$ and $\lambda\in(0,1]$ (cf., e.g.,~\citet[Theorem~19]{inequalities}). This yields
\begin{align}
&\|\partial_{\bsy}^{\boldsymbol m}{\rm tr}\big(J(\cdot,\bsy)^{-1}\partial_{\bsy}^{\boldsymbol e_k}J(\cdot,\bsy)\big)\|_{L^\infty}\notag\\
&\leq C^{|\boldsymbol m|+1}\sigma_{\min}^{-|\boldsymbol m|-1}d^22^{\beta(|\boldsymbol m|+1)}(|\boldsymbol m|!)^\beta\boldsymbol b^{\boldsymbol m+\boldsymbol e_k}.\label{eq:trjinvj}%
\end{align}
Now the induction hypothesis together with the inequality~\eqref{eq:trjinvj} imply that~\eqref{eq:detjind} can be bounded by
\begin{align}
&\|\partial_{\bsy}^{\bsnu+\boldsymbol e_k}\det J(\cdot,\bsy)\|_{L^\infty}\notag\\
&\leq \sum_{\boldsymbol m\leq \bsnu}\binom{\bsnu}{\boldsymbol m}\sigma_{\max}^d \bigg(\frac{2^\beta C}{\sigma_{\min}}\bigg)^{|\boldsymbol m|}\frac{((|\boldsymbol m|+d^2-1)!)^\beta}{((d^2-1)!)^\beta}\boldsymbol b^{\boldsymbol m}\notag\\
&\quad \times C^{|\bsnu|-|\boldsymbol m|+1}\sigma_{\min}^{-|\bsnu|+|\boldsymbol m|-1}d^22^{\beta(|\bsnu|-|\boldsymbol m|+1)}((|\bsnu|-|\boldsymbol m|)!)^\beta\boldsymbol b^{\bsnu-\boldsymbol m+\boldsymbol e_k}\notag\\
&=\sigma_{\max}^d C_{\det J}^{|\bsnu|+1}\boldsymbol b^{\bsnu+\boldsymbol e_k}\frac{d^2}{((d^2-1)!)^\beta}\notag\\
&\quad \times \sum_{\ell=0}^{|\bsnu|}((\ell+d^2-1)!)^\beta((|\bsnu|-\ell)!)^\beta\sum_{\substack{\boldsymbol m\leq\bsnu\notag\\ |\boldsymbol m|=\ell}}\binom{\bsnu}{\boldsymbol m}\notag\\
&\leq\sigma_{\max}^d C_{\det J}^{|\bsnu|+1}\boldsymbol b^{\bsnu+\boldsymbol e_k}\frac{d^2}{((d^2-1)!)^\beta}(|\bsnu|!)^\beta\sum_{\ell=0}^{|\bsnu|}\frac{((\ell+d^2-1)!)^\beta}{(\ell!)^\beta}.\label{eq:pregosp}
\end{align}
Here, we can estimate 
\begin{align*}
\sum_{\ell=0}^{|\bsnu|}\frac{((\ell+d^2-1)!)^\beta}{(\ell!)^\beta}&\leq \bigg(\sum_{\ell=0}^{|\bsnu|}\frac{(\ell+d^2-1)!}{\ell!}\bigg)^\beta.%
\end{align*}
The above sum is \emph{Gosper-summable} (cf., e.g., \cite{aeqb}): by letting $t_{\ell}:=\frac{(\ell+d^2-1)!}{\ell!}$ and $z_{\ell}:=\frac{\ell}{d^2}\frac{(\ell+d^2-1)!}{\ell!}$, there holds $t_{\ell}=z_{\ell+1}-z_{\ell}$ and hence
\begin{align}
\sum_{\ell=0}^{|\bsnu|}\frac{(\ell+d^2-1)!}{\ell!}=z_{|\bsnu|+1}-z_0=\frac{1}{d^2}\frac{(|\bsnu|+d^2)!}{|\bsnu|!}.\label{eq:postgosp}
\end{align}
Combining the estimates~\eqref{eq:pregosp}--\eqref{eq:postgosp} yields
\begin{align*}
\|\partial_{\bsy}^{\bsnu+\boldsymbol e_k}\det J(\cdot,\bsy)\|_{L^\infty}&\leq\sigma_{\max}^d C_{\det J}^{|\bsnu|+1}\boldsymbol b^{\bsnu+\boldsymbol e_k}\frac{((|\bsnu|+d^2)!)^\beta}{((d^2-1)!)^\beta}
\end{align*}
as desired.
\end{proof}

\paragraph{Step 4.} By combining the previous steps, we obtain the parametric regularity bound for the pullback diffusion coefficient.

\begin{lemma}\label{lemma:productreg} Let $\bsy\in U$ and $\bsnu\in\mathscr F$. Under Assumptions~\ref{a1}\,--\,\ref{assump:gevrey}, there holds
$$
\|\partial_{\bsy}^{\bsnu}A(\cdot,\bsy)\|_{L^\infty(D_{\rm ref})}\leq C_{A,1}C_{A,2}^{|\bsnu|}\max\{1,2^{|\bsnu|-1}\}\frac{((|\bsnu|+d^2)!)^\beta}{(d^2!)^\beta}\boldsymbol b^{\bsnu},
$$
where
$$
C_{A,1}=\sigma_{\max}^d \sigma_{\min}^{-2}\quad\text{and}\quad C_{A,2}=C_{\det J}\sigma_{\min}^{-2}C^2 4^\beta.
$$
\end{lemma}
\begin{proof}
By the Leibniz product rule, we obtain that
\begin{align*}
&\|\partial_{\bsy}^{\bsnu}\big(B(\cdot,\bsy)^{-1}\det J(\cdot,\bsy)\big)\|_{L^\infty}\\
&\leq \sum_{\boldsymbol m\leq \bsnu}\binom{\bsnu}{\boldsymbol m}\|\partial_{\bsy}^{\boldsymbol m}\det J(\cdot,\bsy)\|_{L^\infty}\|\partial_{\bsy}^{\bsnu-\boldsymbol m}B(\cdot,\bsy)^{-1}\|_{L^\infty}\\
&\leq \frac{\sigma_{\max}^d}{((d^2-1)!)^\beta}\sum_{\boldsymbol m\leq \bsnu}\binom{\bsnu}{\boldsymbol m}C_{\det J}^{|\boldsymbol m|}((|\boldsymbol m|+d^2-1)!)^\beta\boldsymbol b^{\boldsymbol m}\\
&\quad\quad\quad \times\sigma_{\min}^{-2|\bsnu|+2|\boldsymbol m|-2}C^{2|\bsnu|-2|\boldsymbol m|}((|\bsnu|-|\boldsymbol m|)!)^\beta4^{\beta|\bsnu|-\beta|\boldsymbol m|}\\
&\quad\quad\quad\times\max\{1,2^{|\bsnu|-|\boldsymbol m|-1}\}\boldsymbol b^{\bsnu-\boldsymbol m}\\
&\leq \frac{\sigma_{\max}^d}{((d^2-1)!)^\beta}C_{\det J}^{|\bsnu|}\sigma_{\min}^{-2|\bsnu|-2}\max\{1,C^2\}^{|\bsnu|}4^{\beta|\bsnu|}\max\{1,2^{|\bsnu|-1}\}\boldsymbol b^{\bsnu}\\
&\quad\quad\quad \times\underset{\leq(|\bsnu|!)^\beta\frac{(|\bsnu|+1)^\beta\cdot ((|\bsnu|+d^2)!)^\beta}{d^{2\beta}\cdot ((|\bsnu|+1)!)^\beta}~\text{as above}}{\underbrace{\sum_{\boldsymbol m\leq \bsnu}\binom{\bsnu}{\boldsymbol m}((|\boldsymbol m|+d^2-1)!)^\beta((|\bsnu|-|\boldsymbol m|)!)^\beta}}\\
&=\frac{\sigma_{\max}^d}{(d^2!)^\beta}C_{\det J}^{|\bsnu|}\sigma_{\min}^{-2|\bsnu|-2}\max\{1,C^2\}^{|\bsnu|}4^{\beta|\bsnu|}\\
&\quad\times\max\{1,2^{|\bsnu|-1}\}((|\bsnu|+d^2)!)^\beta\boldsymbol b^{\bsnu},
\end{align*}
as desired.
\end{proof}

\subsection{Regularity of the right-hand side}\label{sec:rhsreg}
\paragraph{Step 5.} We also need to control the parametric regularity of the pullback source term. We have the following result.
\begin{lemma}\label{8} Under Assumptions~\ref{a1}\,--\,\ref{a4}, there holds
$$
\|\partial_{\bsy}^{\bsnu}f(\boldsymbol V(\cdot,\bsy))\|_{L^\infty(D_{\rm ref})}\leq 2^{\beta |\bsnu|-\beta }\max\{1,\|\boldsymbol\rho\|_{\ell^{1/\beta}}^{|\bsnu|}\}C_fC^{|\bsnu|}\boldsymbol b^{\bsnu}(|\bsnu|!)^\beta
$$
for all $\bsnu\in\mathscr F\setminus\{\mathbf 0\}$ and $\bsy\in U$.
\end{lemma}
\begin{proof} By Fa\`a di Bruno's formula~\citep{savits}, there holds
$$
\partial_{\bsy}^{\bsnu}f(\boldsymbol V(\bsx,\bsy))=\sum_{\substack{\boldsymbol\lambda\in\mathbb N_0^d\\ 1\leq |\boldsymbol\lambda|\leq |\bsnu|}}\partial_{\bsx}^{\boldsymbol\lambda}f(\bsx)\bigg|_{\bsx=\boldsymbol V(\bsx,\bsy)}\alpha_{\bsnu,\boldsymbol\lambda}(\bsx,\bsy)\quad\text{for all}~\bsnu\in\mathscr F\setminus\{\mathbf 0\},
$$
where the sequence $(\alpha_{\bsnu,\boldsymbol\lambda}(\bsx,\bsy))_{\bsnu\in\mathscr F,\boldsymbol\lambda\in\mathbb Z^d}$ is defined recursively by setting $\alpha_{\bsnu,\mathbf 0}\equiv \delta_{\bsnu,\mathbf 0}$ for all $\bsnu\in\mathscr F$, $\alpha_{\bsnu,\boldsymbol\lambda}\equiv 0$ if $|\bsnu|<|\boldsymbol\lambda|$ or $\boldsymbol\lambda\not\geq \mathbf 0$, and
$$
\alpha_{\bsnu+\boldsymbol e_j,\boldsymbol\lambda}(\bsx,\bsy)=\sum_{\ell\in{\rm supp}(\boldsymbol\lambda)}\sum_{0\leq |\boldsymbol m|\leq|\bsnu|}\binom{\bsnu}{\boldsymbol m}\partial_{\bsy}^{\boldsymbol m+\boldsymbol e_j}[\boldsymbol V(\bsx,\bsy)]_{\ell}\alpha_{\bsnu-\boldsymbol m,\boldsymbol\lambda-\boldsymbol e_{\ell}}(\bsx,\bsy)
$$
otherwise, where $\boldsymbol e_j\in\mathscr F$ denotes the multi-index whose $j^{\rm th}$ component is 1 and other entries are 0.

By Assumption~\ref{assump:gevrey}, we especially have that
\begin{align*}
&\|\alpha_{\bsnu+\boldsymbol e_j,\boldsymbol\lambda}(\cdot,\bsy)\|_{L^\infty}\\
&\quad\leq C\sum_{\ell\in{\rm supp}(\boldsymbol\lambda)}\sum_{0\leq |\boldsymbol m|\leq|\bsnu|}\binom{\bsnu}{\boldsymbol m}((|\boldsymbol m|+1)!)^\beta\boldsymbol b^{\boldsymbol m+\boldsymbol e_j}\|\alpha_{\bsnu-\boldsymbol m,\boldsymbol\lambda-\boldsymbol e_{\ell}}(\cdot,\bsy)\|_{L^\infty}.
\end{align*}
As a consequence of~\citet[Theorem~5.1]{ks24}, the sequence $\|\alpha_{\bsnu,\boldsymbol\lambda}(\cdot,\bsy)\|_{L^\infty}$ can be bounded by
$$
\|\alpha_{\bsnu,\boldsymbol\lambda}(\cdot,\bsy)\|_{L^\infty}\leq C^{|\bsnu|}\frac{(|\bsnu|!)^\beta ((|\bsnu|-1)!)^\beta}{(\boldsymbol\lambda!)^\beta((|\bsnu|-|\boldsymbol\lambda|)!)^\beta((|\boldsymbol\lambda|-1)!)^\beta}\boldsymbol b^{\bsnu}
$$
and, by employing Assumption~\ref{a4}, we obtain
\begin{align*}
&\|\partial_{\bsy}^{\bsnu}f(\boldsymbol V(\cdot,\bsy))\|_{L^\infty}\\
&\leq C_fC^{|\bsnu|}\boldsymbol b^{\bsnu}(|\bsnu|!)^\beta((|\bsnu|-1)!)^\beta\sum_{\substack{\boldsymbol\lambda\in\mathbb N_0^d\\ 1\leq|\boldsymbol\lambda|\leq|\bsnu|}}\frac{(|\boldsymbol\lambda|!)^\beta}{(\boldsymbol\lambda!)^\beta ((|\bsnu|-|\boldsymbol\lambda|)!)^\beta((|\boldsymbol\lambda|-1)!)^\beta}\boldsymbol\rho^{\boldsymbol\lambda}\\
&=C_fC^{|\bsnu|}\boldsymbol b^{\bsnu}(|\bsnu|!)^\beta((|\bsnu|-1)!)^\beta\sum_{\ell=1}^{|\bsnu|}\frac{(\ell!)^\beta}{((|\bsnu|-\ell)!)^\beta((\ell-1)!)^\beta}\sum_{\substack{\boldsymbol\lambda\in\mathbb N_0^d\\|\boldsymbol\lambda|=\ell}}\frac{1}{(\boldsymbol\lambda!)^\beta}\boldsymbol\rho^{\boldsymbol\lambda}\\
&\leq C_fC^{|\bsnu|}\boldsymbol b^{\bsnu}(|\bsnu|!)^\beta((|\bsnu|-1)!)^\beta\sum_{\ell=1}^{|\bsnu|}\frac{1}{((|\bsnu|-\ell)!)^\beta((\ell-1)!)^\beta}\bigg(\sum_{\substack{\boldsymbol\lambda\in\mathbb N_0^d\\|\boldsymbol\lambda|=\ell}}\frac{\ell!}{(\boldsymbol\lambda!)}\boldsymbol\rho^{\boldsymbol\lambda/\beta}\bigg)^\beta\\
&= C_fC^{|\bsnu|}\boldsymbol b^{\bsnu}(|\bsnu|!)^\beta((|\bsnu|-1)!)^\beta\sum_{\ell=1}^{|\bsnu|}\frac{1}{((|\bsnu|-\ell)!)^\beta((\ell-1)!)^\beta}\bigg(\sum_{j=1}^d \rho_j^{1/\beta}\bigg)^{\beta\ell}\\
&\leq \max\{1,\|\boldsymbol\rho\|_{\ell^{1/\beta}}^{|\bsnu|}\}C_fC^{|\bsnu|}\boldsymbol b^{\bsnu}(|\bsnu|!)^\beta ((|\bsnu|-1)!)^\beta \sum_{\ell=1}^{|\bsnu|}\frac{1}{((|\bsnu|-\ell)!)^\beta ((\ell-1)!)^\beta}\\
&= 2^{\beta |\bsnu|-\beta}\max\{1,\|\boldsymbol\rho\|_{\ell^{1/\beta}}^{|\bsnu|}\}C_fC^{|\bsnu|}\boldsymbol b^{\bsnu}(|\bsnu|!)^\beta,
\end{align*}
where we used the multinomial theorem as well as the following bound (cf., e.g., \citet[Lemma~A.1]{ks24})
$$
\sum_{\ell=1}^v \frac{1}{((v-\ell)!)^\beta ((\ell-1)!)^\beta}\leq \bigg(\sum_{\ell=1}^v \frac{1}{(v-\ell)! (\ell-1)!}\bigg)^\beta=\frac{2^{\beta \ell-\beta }}{((v-1)!)^\beta}.
$$This proves the assertion.\end{proof}

\begin{lemma}Let $\bsnu\in\mathscr F\setminus\{\mathbf 0\}$ and $\bsy\in U$. Under Assumptions~\ref{a1}\,--\,\ref{a4}, there holds
$$
\|\partial_{\bsy}^{\bsnu}f_{\rm ref}(\cdot,\bsy)\|_{L^\infty(D_{\rm ref})}\leq C_{f_{\rm ref},1}C_{f_{\rm ref},2}^{|\bsnu|}\max\{1,\|\boldsymbol\rho\|_{\ell^{1/\beta}}^{|\bsnu|}\}\frac{((|\bsnu|+d^2)!)^\beta}{(d^2!)^\beta}\boldsymbol b^{\bsnu},
$$
where
$$
C_{f_{\rm ref},1}=\sigma_{\max}^d 2^{-\beta}C_f\quad\text{and}\quad C_{f_{\rm ref},2}=\frac{2^\beta C^2}{\sigma_{\min}}2^{\beta}.
$$
\end{lemma}
\begin{proof}
The claim follows by an application of the Leibniz product rule:
\begin{align*}
\|\partial_{\bsy}^{\bsnu}f_{\rm ref}(\cdot,\bsy)\|_{L^\infty}\leq \sum_{\boldsymbol m\leq \bsnu}\binom{\bsnu}{\boldsymbol m}\|\partial_{\bsy}^{\boldsymbol m}\det J(\cdot,\bsy)\|_{L^\infty}\|\partial_{\bsy}^{\bsnu-\boldsymbol m}f(\boldsymbol V(\cdot,\bsy))\|_{L^\infty},
\end{align*}
where the result follows using the same steps as the proof of Lemma~\ref{lemma:productreg}.
\end{proof}

\subsection{Parametric regularity of the solution to the pullback Poisson equation}
We have the following.
\begin{theorem}\label{RegBoundPDEsolStat}
Let $\bsnu\in\mathscr F$ and $\bsy\in U$. Under Assumptions~\ref{a1}\,--\,\ref{a4}, the solution to~\eqref{eq:weakform} satisfies
\begin{align*}
\|\partial_{\bsy}^{\bsnu}\widehat u(\cdot,\bsy)\|_{H_0^1(D_{\rm ref})}\leq C_{\widehat u,1}C_{\widehat u,2}^{|\bsnu|}(|\bsnu|!)^{\beta}\boldsymbol b^{\bsnu}
\end{align*}
where
\begin{align*}
&C_{\widehat u,1}=1+\frac{\sigma_{\max}^2}{\sigma_{\min}^d}|D_{\rm ref}|^{1/2}C_{D_{\rm ref}}C_{f_{\rm ref},1},\\
&C_{\widehat u,2}=\max\{c_0,c_1,\widetilde c_1,\widetilde c_2\}^2 ((d^2)!)^{\beta}2^{\beta(d^2+1)+1},
\end{align*}
with \begin{align*}
c_0=\frac{C_{A,1}}{\sigma_{\min}^d((d^2)!)^\beta},\quad c_1=2C_{A,2},\quad &\widetilde c_0 = \frac{|D_{\rm ref}|^{1/2}C_{D_{\rm ref}}C_{f_{\rm ref},1}}{\sigma_{\min}^d ((d^2)!)^\beta},\quad\text{and}\\
&\widetilde c_1=C_{f_{\rm ref},2}\max\{1,\|\boldsymbol\rho\|_{\ell^{1/\beta}}\}.
\end{align*}
\end{theorem}
\begin{proof}
By differentiating the weak formulation~\eqref{eq:weakform} on both sides with respect to $\bsy\in U$, we obtain
\begin{align*}
&\sum_{\boldsymbol m\leq\bsnu}\binom{\bsnu}{\boldsymbol m}\int_{D_{\rm ref}} \partial_{\bsy}^{\boldsymbol m}A(\bsx,\bsy)\nabla \partial_{\bsy}^{\bsnu-\boldsymbol m}\widehat u(\bsx,\bsy)\cdot \nabla \widehat v(\bsx)\,{\rm d}\bsx\\
&\quad=\int_{D_{\rm ref}}\partial_{\bsy}^{\bsnu}f_{\rm ref}(\cdot,\bsy)\widehat v(\bsx)\,{\rm d}\bsx.
\end{align*}
By using standard estimates and testing again with $\widehat v=\partial_{\bsy}^{\bsnu}\widehat u(\cdot,\bsy)\in H_0^1(D_{\rm ref})$, we obtain the recurrence relation
\begin{align*}
\|\partial_{\bsy}^{\bsnu}\widehat u(\cdot,\bsy)\|_{H_0^1}\leq &\frac{1}{\sigma_{\min}^d}\sum_{\substack{\boldsymbol m\leq\bsnu\\ \boldsymbol m\neq\mathbf 0}}\binom{\bsnu}{\boldsymbol m}\|\partial_{\bsy}^{\boldsymbol m}A(\cdot,\bsy)\|_{L^\infty}\| \partial_{\bsy}^{\bsnu-\boldsymbol m}\widehat u(\cdot,\bsy)\|_{H_0^1}\\
&+\frac{|D_{\rm ref}|^{1/2}C_{D_{\rm ref}}}{\sigma_{\min}^d}\|\partial_{\bsy}^{\bsnu}f_{\rm ref}\|_{L^\infty},
\end{align*}
where $|D_{\rm ref}|=\int_{D_{\rm ref}}\,{\rm d}\bsx$ and $C_{D_{\rm ref}}>0$ denotes the Poincar\'e constant of domain $D_{\rm ref}$.

We write the parametric regularity bound for $A(\bsx,\bsy)$ as
$$
\|\partial_{\bsy}^{\bsnu}A(\cdot,\bsy)\|_{L^\infty}\leq \sigma_{\min}^dc_0c_1^{|\bsnu|}((|\bsnu|+d^2)!)^{\beta}\boldsymbol b^{\bsnu},
$$
where $c_0=\frac{C_{A,1}}{\sigma_{\min}^d((d^2)!)^\beta}$ and $c_1=2C_{A,2}$, and the parametric regularity bound for $f_{\rm ref}(\bsx,\bsy)=f(\boldsymbol V(\bsx,\bsy))\det J(\bsx,\bsy)$ as
$$
\|\partial_{\bsy}^{\bsnu}f_{\rm ref}(\cdot,\bsy)\|_{L^\infty}\leq \frac{\sigma_{\min}^d \widetilde c_0}{|D_{\rm ref}|^{1/2}C_{D_{\rm ref}}}\widetilde c_1^{|\bsnu|}((|\bsnu|+d^2)!)^{\beta}\boldsymbol b^{\bsnu},
$$
where $\widetilde c_0 = \frac{|D_{\rm ref}|^{1/2}C_{D_{\rm ref}}C_{f_{\rm ref},1}}{\sigma_{\min}^d ((d^2)!)^\beta}$ and $\widetilde c_1=C_{f_{\rm ref},2}\max\{1,\|\boldsymbol\rho\|_{\ell^{1/\beta}}\}$. Then we obtain the recursive bound
\begin{align*}
\|\widehat u(\cdot,\bsy)\|_{H_0^1}&\leq C_0,\\
\|\partial_{\bsy}^{\bsnu}\widehat u(\cdot,\bsy)\|_{H_0^1}&\leq  C\sum_{\substack{\boldsymbol m\leq\bsnu\\ \boldsymbol m\neq\mathbf 0}}\binom{\bsnu}{\boldsymbol m}C^{|\boldsymbol m|}((|\boldsymbol m|+d^2)!)^{\beta}\boldsymbol b^{\boldsymbol m}\|\partial_{\bsy}^{\bsnu-\boldsymbol m}\widehat u(\cdot,\bsy)\|_{H_0^1}\\
&+ C C^{|\bsnu|}((|\bsnu|+d^2)!)^{\beta}\boldsymbol b^{\boldsymbol\nu},
\end{align*}
where $C_0=\frac{\sigma_{\max}^2}{\sigma_{\min}^d}|D_{\rm ref}|^{1/2}C_{D_{\rm ref}}C_{f_{\rm ref},1}$ and $C=\max\{c_0, c_1,\widetilde c_0,\widetilde c_1\}$.
By Lemma \ref{superlemma}, we obtain the inductive bound
\begin{align*}
&\|\partial_{\bsy}^{\bsnu}\widehat u(\cdot,\bsy)\|_{H_0^1}\leq (1+ C_0)C^{|\bsnu|}(|\bsnu|!)^{\beta}\widetilde \tau_{|\bsnu|,\beta,d^2}\boldsymbol b^{\bsnu},
\end{align*}
where $\widetilde\tau$ is defined by the recursion~\eqref{eq:tautildeseq}. Especially, we can bound
$$
\widetilde \tau_{\nu,\beta,d^2}\leq C^{\nu}\tau_{\nu,\beta,d^2}
$$
with $\tau$ defined by~\eqref{eq:tauseq}. The claim follows by an application of Lemma~\ref{lemma:taubound}.
\end{proof}
\subsection{Parametric regularity of the solution to the pullback heat equation}

In contrast to the pullback Poisson equation, the pullback heat equation contains several additional terms which need to be accounted for further complicating the analysis. Thus we approach the analysis step-by-step as follows:
\begin{addmargin}[3em]{0em}
\begin{enumerate}
\item[\em Step 1:] We begin by differentiating the variational formulation~\eqref{spacetimepullback} with respect to the parametric variable.
\item[\em Step 2:] In order to obtain bounds for the individual terms, we need to fix the test function satisfying certain conditions.
\item[\em Step 3:] Using steps 1 and 2, we derive a lower bound on the left-hand side of the differentiated variational formulation.
\item[\em Step 4:] We proceed to derive an upper bound for the right-hand side of the differentiated variational formulation.
\item[\em Step 5:] We shall additionally require a parametric regularity bound for the pullback initial condition $\widehat u_0$.
\item[\em Step 6:] Finally, we combine the bounds derived in steps 3--5 and derive an inductive bound for the solution to the pullback heat equation.
\end{enumerate}
\end{addmargin}
\paragraph{Step 1.} As shown in Lemma \ref{L2reg} we have that $\partial_{\bsy}^{\bsnu}\widehat u(\cdot,\cdot,\bsy)\in L^2(I;H_0^1(D_{\rm ref}))$ for all $\bsnu\in\mathscr F$. By differentiating the variational formulation~\eqref{spacetimepullback} on both sides with respect to $\bsy$ and applying the Leibniz product rule, we obtain
\begin{align*}
&\sum_{\boldsymbol m\leq\bsnu}\binom{\bsnu}{\boldsymbol m}\bigg[\int_I \int_{D_{\rm ref}}\tfrac{\partial}{\partial t}(\partial_\bsy^{\bsnu-\boldsymbol m}\widehat u(\cdot,t,\bsy))\widehat v_1(\cdot,t)\partial_\bsy^{\boldsymbol m}\det J(\cdot,\bsy)\,{\rm d}\bsx\,{\rm d}t\\
&\quad +\int_I\int_{D_{\rm ref}}\partial_\bsy^{\boldsymbol m}A(\bsx,\bsy)\nabla \partial_{\bsy}^{\bsnu-\boldsymbol m}\widehat u(\bsx,t,\bsy)\cdot \nabla \widehat v_1(\bsx,t)\,{\rm d}\bsx\,{\rm d}t\\
&\quad +\int_{D_{\rm ref}}\partial_\bsy^{\bsnu-\boldsymbol m}\widehat u(\bsx,0,\bsy)\widehat v_2(\bsx)\partial_{\bsy}^{\boldsymbol m}\det J(\bsx,\bsy)\,{\rm d}\bsx\bigg]\\
&=\int_I \int_{D_{\rm ref}}\partial_{\bsy}^{\bsnu}f_{\rm ref}(\cdot,t,\bsy)\widehat v_1\,{\rm d}\bsx\,{\rm d}t\\
&\quad + \sum_{\boldsymbol m\leq \bsnu}\binom{\bsnu}{\boldsymbol m}\int_{D_{\rm ref}}\partial_{\bsy}^{\bsnu-\boldsymbol m}\widehat u_0(\bsx,\bsy)\widehat v_2(\bsx)\partial_{\bsy}^{\boldsymbol m}\det J(\bsx,\bsy)\,{\rm d}\bsx,
\end{align*}
where $\widehat u_0(\bsx,\bsy)=u_0(\boldsymbol V(\bsx,\bsy),\bsy)$ for $\bsx\in D_{\rm ref}$, $\bsy\in U$. 
Separating out the $\boldsymbol m=\mathbf 0$ term, we obtain
\begin{align*}
&\int_I \int_{D_{\rm ref}} \tfrac{\partial}{\partial t}(\partial_\bsy^{\bsnu}\widehat u(\cdot,t,\bsy))\widehat v_1(\cdot,t)\det J(\cdot,\bsy)\,{\rm d}\bsx\,{\rm d}t\\
&\quad + \int_I \int_{D_{\rm ref}}A(\bsx,\bsy)\nabla \partial_\bsy^{\bsnu}\widehat u(\bsx,t,\bsy)\cdot \nabla \widehat v_1(\bsx,t)\,{\rm d}\bsx\,{\rm d}t\\
&\quad + \int_{D_{\rm ref}}\partial_{\bsy}^{\bsnu}\widehat u(\bsx,0,\bsy)\widehat v_2(\bsx)\det J(\bsx,\bsy)\,{\rm d}\bsx\\
&=-\sum_{\mathbf 0\neq\boldsymbol m\leq\bsnu}\binom{\bsnu}{\boldsymbol m}\bigg[\int_I \int_{D_{\rm ref}}\tfrac{\partial}{\partial t}\partial_\bsy^{\bsnu-\boldsymbol m}\widehat u(\cdot,t,\bsy)\widehat v_1(\cdot,t)\partial_\bsy^{\boldsymbol m}\det J(\cdot,\bsy)\,{\rm d}\bsx\,{\rm d}t\\
&\quad +\int_I \int_{D_{\rm ref}}\partial_\bsy^{\boldsymbol m}A(\bsx,\bsy)\nabla \partial_{\bsy}^{\bsnu-\boldsymbol m}\widehat u(\bsx,t,\bsy)\cdot \nabla \widehat v_1(\bsx,t)\,{\rm d}\bsx\,{\rm d}t\\
&\quad +\int_{D_{\rm ref}}\partial_\bsy^{\bsnu-\boldsymbol m}\widehat u(\bsx,0,\bsy)\widehat v_2(\bsx)\partial_{\bsy}^{\boldsymbol m}\det J(\bsx,\bsy)\,{\rm d}\bsx\bigg]\\
&\quad +\int_I\int_{D_{\rm ref}}  \partial_{\bsy}^{\bsnu}f_{\rm ref}(\cdot,t,\bsy)\widehat v_1\,{\rm d}\bsx\,{\rm d}t\\
&\quad+\sum_{\boldsymbol m\leq \bsnu}\binom{\bsnu}{\boldsymbol m}\int_{D_{\rm ref}}\partial_{\bsy}^{\bsnu-\boldsymbol m}\widehat u_0(\bsx,\bsy)\widehat v_2(\bsx)\partial_{\bsy}^{\boldsymbol m}\det J(\bsx,\bsy)\,{\rm d}\bsx.
\end{align*}
\paragraph{Step 2.}
Let $\mathcal S(\bsy)\!:H_0^1(D_{\rm ref})\to H^{-1}(D_{\rm ref})$ denote the differential operator
$$
\langle \mathcal S(\bsy)\widehat v,\widehat w\rangle_{H^{-1}(D_{\rm ref}),H_0^1(D_{\rm ref})}:=\int_{D_{\rm ref}}A(\bsx,\bsy)\nabla \widehat v(\bsx)\cdot \nabla \widehat w(\bsx)\,{\rm d}\bsx.
$$
We consider the test function 
\begin{align*}
    (\widehat v_1,\widehat v_2)=(\widehat z(\cdot,\cdot,\bsy)+\partial_{\bsy}^{\bsnu}\widehat u(\cdot,\cdot,\bsy),\partial_{\bsy}^{\bsnu}\widehat u(\bsx,0,\bsy))\in\mathcal Y,
\end{align*}
where $\widehat z(\bsx,t,\bsy)=(-\Delta)^{-1}(\det J(\bsx,\bsy)\frac{\partial}{\partial t}\partial_{\bsy}^{\bsnu}\widehat u\big)$ and the %
 negative Laplacian
 \begin{align*}
  -\Delta\!:H_0^1(D_{\rm ref})\cap H^2(D_{\rm ref})\to L^2(D_{\rm ref}).   
 \end{align*}
 Especially, there hold 
\begin{align*}\|\nabla v\|_{H_0^1(D_{\rm ref})}\geq \frac{1}{C_{D_{\rm ref}}}\|v\|_{L^2(D_{\rm ref})}\end{align*}
and 
\begin{align*}\|-\Delta v\|_{L^2(D_{\rm ref})}\leq \|v\|_{H_0^1(D_{\rm ref})\cap H^2(D_{\rm ref})}
\end{align*}
for all $v\in H_0^1(D_{\rm ref})\cap H^2(D_{\rm ref})$.

Hence
\begin{align*}
&\int_{D_{\rm ref}}\tfrac{\partial}{\partial t}\partial_{\bsy}^{\bsnu}\widehat u\widehat z\det J\,{\rm d}\bsx=\langle -\Delta \widehat z,\widehat z\rangle_{L^2(D_{\rm ref})}=\langle \nabla  \widehat z,\nabla \widehat z\rangle_{L^2(D_{\rm ref})}\\
&\geq \frac{1}{C_{D_{\rm ref}}^2}\|\widehat z\|_{H_0^1(D_{\rm ref})}^2\geq\frac{1}{C_{D_{\rm ref}}^2}\|\det J(\cdot,\bsy)\tfrac{\partial}{\partial t}\partial_{\bsy}^{\bsnu}\widehat u\|_{L^2(D_{\rm ref})}^2\\
&\geq \frac{\sigma_{\min}^{2d}}{C_{D_{\rm ref}}^2}\|\tfrac{\partial}{\partial t}\partial_{\bsy}^{\bsnu}\widehat u\|_{L^2(D_{\rm ref})}^2;
\end{align*}
see also~\citet[pg.~1315]{schwabstevenson}.

Plugging in the aforementioned test function, the variational formulation becomes
\begin{align*}
&\int_I \int_{D_{\rm ref}} \tfrac{\partial}{\partial t}\partial_\bsy^{\bsnu}\widehat u(\bsx,t,\bsy) \widehat z(\bsx,t,\bsy)\det J(\cdot,\bsy)\,{\rm d}\bsx\,{\rm d}t\\
&\quad+\int_I \int_{D_{\rm ref}} \tfrac{\partial}{\partial t}\partial_\bsy^{\bsnu}\widehat u(\bsx,t,\bsy) \partial_{\bsy}^{\bsnu}\widehat u(\bsx,t,\bsy)\det J(\cdot,\bsy)\,{\rm d}\bsx\,{\rm d}t\\
&\quad + \int_I \int_{D_{\rm ref}}A(\bsx,\bsy)\nabla \partial_\bsy^{\bsnu}\widehat u(\bsx,t,\bsy)\cdot \nabla \widehat z(\bsx,t,\bsy)\,{\rm d}\bsx\,{\rm d}t\\
&\quad+\int_I \int_{D_{\rm ref}}A(\bsx,\bsy)\nabla \partial_\bsy^{\bsnu}\widehat u(\bsx,t,\bsy)\cdot \nabla \partial_\bsy^{\bsnu}\widehat u(\bsx,t,\bsy)\,{\rm d}\bsx\,{\rm d}t\\
&\quad + \int_{D_{\rm ref}}|\partial_{\bsy}^{\bsnu}\widehat u(\bsx,0,\bsy)|^2\det J(\bsx,\bsy)\,{\rm d}\bsx\\
&=-\sum_{\mathbf 0\neq\boldsymbol m\leq\bsnu}\binom{\bsnu}{\boldsymbol m}\bigg[\int_I \int_{D_{\rm ref}}\tfrac{\partial}{\partial t}\partial_\bsy^{\bsnu-\boldsymbol m}\widehat u(\cdot,t,\bsy)\widehat z(\cdot,t)\partial_\bsy^{\boldsymbol m}\det J(\cdot,\bsy)\,{\rm d}\bsx\,{\rm d}t\\
&\quad+\int_I \int_{D_{\rm ref}}\tfrac{\partial}{\partial t}\partial_\bsy^{\bsnu-\boldsymbol m}\widehat u(\cdot,t,\bsy)\partial_{\bsy}^{\bsnu}\widehat u(\cdot,t)\partial_\bsy^{\boldsymbol m}\det J(\cdot,\bsy)\,{\rm d}\bsx\,{\rm d}t\\
&\quad +\int_I \int_{D_{\rm ref}}\partial_\bsy^{\boldsymbol m}A(\bsx,\bsy)\nabla \partial_{\bsy}^{\bsnu-\boldsymbol m}\widehat u(\bsx,t,\bsy)\cdot \nabla \widehat z(\bsx,t)\,{\rm d}\bsx\,{\rm d}t\\
&\quad+\int_I \int_{D_{\rm ref}}\partial_\bsy^{\boldsymbol m}A(\bsx,\bsy)\nabla \partial_{\bsy}^{\bsnu-\boldsymbol m}\widehat u(\bsx,t,\bsy)\cdot \nabla \partial_{\bsy}^{\bsnu}\widehat u(\bsx,t)\,{\rm d}\bsx\,{\rm d}t\\
&\quad +\int_{D_{\rm ref}}\partial_\bsy^{\bsnu-\boldsymbol m}\widehat u(\bsx,0,\bsy)\partial_{\bsy}^{\bsnu}\widehat u(\bsx,0,\bsy)\partial_{\bsy}^{\boldsymbol m}\det J(\bsx,\bsy)\,{\rm d}\bsx\bigg]\\
&\quad +\int_I\int_{D_{\rm ref}}  \partial_{\bsy}^{\bsnu}f_{\rm ref}(\cdot,t,\bsy)\widehat z\,{\rm d}\bsx\,{\rm d}t+\int_I\int_{D_{\rm ref}}  \partial_{\bsy}^{\bsnu}f_{\rm ref}(\cdot,t,\bsy)\partial_{\bsy}^{\bsnu}\widehat u\,{\rm d}\bsx\,{\rm d}t\\
&\quad +\sum_{\boldsymbol m\leq \bsnu}\binom{\bsnu}{\boldsymbol m}\int_{D_{\rm ref}}\partial_{\bsy}^{\bsnu-\boldsymbol m}\widehat u_0(\bsx,\bsy)\partial_{\bsy}^{\bsnu}\widehat u(\bsx,0,\bsy)\partial_{\bsy}^{\boldsymbol m}\det J(\bsx,\bsy)\,{\rm d}\bsx.
\end{align*}
\paragraph{Step 3.} The left-hand side can be bounded from below by noting that, by definition of $\widehat z$, there holds
\begin{align*}
\int_{D_{\rm ref}}A\nabla \partial_{\bsy}^{\bsnu}\widehat u\cdot \nabla \widehat z\,{\rm d}\bsx=\int_{D_{\rm ref}}\partial_{\bsy}^{\bsnu}\widehat u\tfrac{\partial}{\partial t}\partial_{\bsy}^{\bsnu}\widehat u\det J(\bsx,\bsy)\,{\rm d}\bsx,
\end{align*}
and thus
\begin{align*}
&\int_I \int_{D_{\rm ref}} \tfrac{\partial}{\partial t}\partial_\bsy^{\bsnu}\widehat u(\bsx,t,\bsy) \widehat z(\bsx,t,\bsy)\det J(\cdot,\bsy)\,{\rm d}\bsx\,{\rm d}t\\
&\quad+\int_I \int_{D_{\rm ref}} \tfrac{\partial}{\partial t}\partial_\bsy^{\bsnu}\widehat u(\bsx,t,\bsy) \partial_{\bsy}^{\bsnu}\widehat u(\bsx,t,\bsy)\det J(\cdot,\bsy)\,{\rm d}\bsx\,{\rm d}t\\
&\quad + \int_I \int_{D_{\rm ref}}A(\bsx,\bsy)\nabla \partial_\bsy^{\bsnu}\widehat u(\bsx,t,\bsy)\cdot \nabla \widehat z(\bsx,t,\bsy)\,{\rm d}\bsx\,{\rm d}t\\
&\quad+\int_I \int_{D_{\rm ref}}A(\bsx,\bsy)\nabla \partial_\bsy^{\bsnu}\widehat u(\bsx,t,\bsy)\cdot \nabla \partial_\bsy^{\bsnu}\widehat u(\bsx,t,\bsy)\,{\rm d}\bsx\,{\rm d}t\\
&\quad + \int_{D_{\rm ref}}|\partial_{\bsy}^{\bsnu}\widehat u(\bsx,0,\bsy)|^2\det J(\bsx,\bsy)\,{\rm d}\bsx\\
&\geq \frac{\sigma_{\min}^{2d}}{C_{D_{\rm ref}}^2}\|\tfrac{\partial}{\partial t}\partial_{\bsy}^{\bsnu}\widehat u\|_{\mathcal L^2}^2+\int_I \partial_t \int_{D_{\rm ref}}|\partial_{\bsy}^{\bsnu}\widehat u|^2\det J\,{\rm d}\bsx\,{\rm d}t \\
&\quad+ \int_I \int_{D_{\rm ref}}A(\bsx,\bsy)\nabla \partial_\bsy^{\bsnu}\widehat u(\bsx,t,\bsy)\cdot \nabla \partial_\bsy^{\bsnu}\widehat u(\bsx,t,\bsy)\,{\rm d}\bsx\,{\rm d}t\\
&\quad + \int_{D_{\rm ref}}|\partial_{\bsy}^{\bsnu}\widehat u(\bsx,0,\bsy)|^2\det J(\bsx,\bsy)\,{\rm d}\bsx\\
&=\frac{\sigma_{\min}^{2d}}{C_{D_{\rm ref}}^2}\|\tfrac{\partial}{\partial t}\partial_{\bsy}^{\bsnu}\widehat u\|_{\mathcal L^2}^2+ \int_I \int_{D_{\rm ref}}A(\bsx,\bsy)\nabla \partial_\bsy^{\bsnu}\widehat u(\bsx,t,\bsy)\cdot \nabla \partial_\bsy^{\bsnu}\widehat u(\bsx,t,\bsy)\,{\rm d}\bsx\,{\rm d}t\\
&\quad + \int_{D_{\rm ref}}|\partial_{\bsy}^{\bsnu}\widehat u(\bsx,T,\bsy)|^2\det J(\bsx,\bsy)\,{\rm d}\bsx\\
&\geq \frac{\sigma_{\min}^{2d}}{C_{D_{\rm ref}}^2}\|\tfrac{\partial}{\partial t}\partial_{\bsy}^{\bsnu}\widehat u\|_{\mathcal L^2}^2 + \frac{\sigma_{\min}^d}{\sigma_{\max}^2}\|\partial_\bsy^{\bsnu}\widehat u(\cdot,\cdot,\bsy)\|_{L^2(I;H_0^1(D_{\rm ref}))}^2\\
&\quad+ \sigma_{\min}^d\|\partial_{\bsy}^{\bsnu}\widehat u(\cdot,T,\bsy)\|_{L^2(D_{\rm ref})}^2\\
&\geq \min\bigg\{\frac{\sigma_{\min}^{2d}}{C_{D_{\rm ref}}^2},\frac{\sigma_{\min}^d}{\sigma_{\max}^2}\bigg\}\|\partial_{\bsy}^{\bsnu}\widehat u(\cdot,\cdot,\bsy)\|_{\mathcal X}^2.
\end{align*}
\paragraph{Step 4. }The right-hand side can be bounded from above
\begin{align*}
&\sum_{\mathbf 0\neq \boldsymbol m\leq \bsnu}\binom{\bsnu}{\boldsymbol m}\bigg[\|\tfrac{\partial}{\partial t}\partial_{\bsy}^{\bsnu-\boldsymbol m}\widehat u(\cdot,\cdot,\bsy)\|_{\mathcal L^2}\|\widehat z(\cdot,\cdot,\bsy)\|_{\mathcal L^2}\sigma_{\max}^d C_{\det J}^{|\boldsymbol m|}\frac{((|\boldsymbol m|+d^2-1)!)^\beta}{((d^2-1)!)^\beta}\\
&\quad+ \|\tfrac{\partial}{\partial t}\partial_{\bsy}^{\bsnu-\boldsymbol m}\widehat u(\cdot,\cdot,\bsy)\|_{\mathcal L^2}\|\partial_{\bsy}^{\bsnu}\widehat u(\cdot,\cdot,\bsy)\|_{\mathcal L^2}\sigma_{\max}^d C_{\det J}^{|\boldsymbol m|}\frac{((|\boldsymbol m|+d^2-1)!)^\beta}{((d^2-1)!)^\beta}\\
&\quad +\|\partial_{\bsy}^{\bsnu-\boldsymbol m}\widehat u(\cdot,\cdot,\bsy)\|_{L^2(I;H_0^1(D_{\rm ref}))}\|\widehat z(\cdot,\cdot,\bsy)\|_{L^2(I;H_0^1(D_{\rm ref}))}\\
&\quad\quad\quad\times C_{A,1}C_{A,2}^{|\boldsymbol m|}\max\{1,2^{|\boldsymbol m|-1}\}\frac{((|\boldsymbol m|+d^2)!)^\beta}{(d^2!)^\beta}\\
&\quad +\|\partial_{\bsy}^{\bsnu-\boldsymbol m}\widehat u(\cdot,\cdot,\bsy)\|_{L^2(I;H_0^1(D_{\rm ref}))}\|\partial_{\bsy}^{\boldsymbol \nu}\widehat u(\cdot,\cdot,\bsy)\|_{L^2(I;H_0^1(D_{\rm ref}))}\\
&\quad\quad\quad\times C_{A,1}C_{A,2}^{|\boldsymbol m|}\max\{1,2^{|\boldsymbol m|-1}\}\frac{((|\boldsymbol m|+d^2)!)^\beta}{(d^2!)^\beta}\\
&\quad +\|\partial_{\bsy}^{\bsnu-\boldsymbol m}\widehat u(\cdot,0,\bsy)\|_{L^2(D_{\rm ref})}\|\partial_{\bsy}^{\bsnu}\widehat u(\cdot,0,\bsy)\|_{L^2(D_{\rm ref})}\\
&\quad\quad\quad\times\sigma_{\max}^d C_{\det J}^{|\boldsymbol m|}\frac{((|\boldsymbol m|+d^2-1)!)^\beta}{((d^2-1)!)^\beta}\bigg]\boldsymbol b^{\boldsymbol m}\\
&\quad +\|\widehat z(\cdot,\cdot,\bsy)\|_{L^2(I\times D_{\rm ref})} C_{f_{\rm ref},1}C_{f_{\rm ref},2}^{|\bsnu|}\max\{1,\|\boldsymbol\rho\|_{\ell^{1/\beta}}^{|\bsnu|}\}\frac{((|\bsnu|+d^2)!)^\beta}{(d^2!)^\beta}\boldsymbol b^{\bsnu}\\
&\quad +\|\partial_{\bsy}^{\bsnu}\widehat u(\cdot,\cdot,\bsy)\|_{L^2(I\times D_{\rm ref})} C_{f_{\rm ref},1}C_{f_{\rm ref},2}^{|\bsnu|}\max\{1,\|\boldsymbol\rho\|_{\ell^{1/\beta}}^{|\bsnu|}\}\frac{((|\bsnu|+d^2)!)^\beta}{(d^2!)^\beta}\boldsymbol b^{\bsnu}\\
&\quad +\sum_{\boldsymbol m\leq \bsnu}\binom{\bsnu}{\boldsymbol m}\|\partial_{\bsy}^{\bsnu-\boldsymbol m}\widehat u_0\|_{L^2(D_{\rm ref})}\|\partial_{\bsy}^{\bsnu}\widehat u(\cdot,0,\bsy)\|_{L^2(D_{\rm ref})}\\
&\quad\quad\quad\times\sigma_{\max}^d C_{\det J}^{|\boldsymbol m|}\frac{((|\boldsymbol m|+d^2-1)!)^\beta}{((d^2-1)!)^\beta}\boldsymbol b^{\boldsymbol m}.
\end{align*}
By~\citet[Theorem~5.9.3]{evans}, there holds for some constant $M>0$
\begin{align}\notag
\max_{t\in I}\|v(\cdot,t)\|_{L^2(D_{\rm ref})}&\leq \frac{M}{\sqrt 2}(\|v\|_{L^2(I;H_0^1(D_{\rm ref}))}+\|\tfrac{\partial}{\partial t}v\|_{L^2(I;H^{-1}(D_{\rm ref}))})\\
&\leq M\|v\|_{\mathcal X},\label{eq:first}
\end{align}
and there exists 
$C_{\Delta,\max}>0$ such that
\begin{align*}
 \|(-\Delta)^{-1}v\|_{H_0^1(D_{\rm ref})\cap H^2(D_{\rm ref})}\leq \frac{C_{\Delta,\max}}{M\sigma_{\max}^d\sqrt{T}}\|v\|_{L^2(D_{\rm ref})}   
\end{align*}
for all $v\in L^2(D_{\rm ref})$. In particular, we have that
\begin{align}\label{eq:second}
\|\widehat z\|_{L^2(I\times D_{\rm ref})}\leq \frac{C_{\Delta,\max}}{M\sigma_{\max}^d\sqrt{T}}\|\det J\partial_{\bsy}^{\bsnu}\partial_t \widehat u\|_{\mathcal L^2}\leq C_{\Delta,\max}\|\partial_{\bsy}^{\bsnu}\widehat u\|_{\mathcal X}.
\end{align}
Moreover, since we have that 
$$
-\Delta \widehat z=\det J \partial_\bsy^{\bsnu}\widehat u\quad \text{in $L^2(D_{\rm ref})$},
$$
we infer
$$
\int_{D_{\rm ref}}\nabla \widehat z\cdot \nabla \widehat v\,{\rm d}\bsx=\int_{D_{\rm ref}} (\partial_{\bsy}^{\bsnu}\widehat u) \widehat v\det J\,{\rm d}\bsx\quad\text{for all}~\widehat v\in H_0^1(D_{\rm ref}).
$$
In consequence,
\begin{align*}
\|\nabla \widehat z\|_{L^2(D_{\rm ref})}^2&\leq \sigma_{\max}^d\|\partial_{\bsy}^{\bsnu}\widehat u\|_{L^2(D_{\rm ref})}\|\widehat z\|_{L^2(D_{\rm ref})}\\
&\leq \sigma_{\max}^dC_{D_{\rm ref}}\|\partial_{\bsy}^{\bsnu}\widehat u\|_{L^2(D_{\rm ref})}\|\nabla\widehat z\|_{L^2(D_{\rm ref})},
\end{align*}
whence
\begin{align}\label{eq:third}
\|\nabla \widehat z\|_{L^2(D_{\rm ref})}\leq \sigma_{\max}^dC_{D_{\rm ref}}\|\partial_{\bsy}^{\bsnu}\widehat u\|_{L^2(D_{\rm ref})}.
\end{align}

Hence, the right-hand side has the upper bound 
\begin{align*}
&\sum_{\mathbf 0\neq \boldsymbol m\leq \bsnu}\binom{\bsnu}{\boldsymbol m}\|\partial_{\bsy}^{\bsnu}\widehat u(\cdot,\cdot,\bsy)\|_{\mathcal X}\boldsymbol b^{\boldsymbol m}\\
&\quad \times \bigg[\|\partial_{\bsy}^{\bsnu-\boldsymbol m}\widehat u(\cdot,\cdot,\bsy)\|_{\mathcal X}C_{\Delta,\max}\sigma_{\max}^{2d} C_{\det J}^{|\boldsymbol m|}\frac{((|\boldsymbol m|+d^2-1)!)^\beta}{((d^2-1)!)^\beta}\\
&\quad +\|\partial_{\bsy}^{\bsnu-\boldsymbol m}\widehat u(\cdot,\cdot,\bsy)\|_{\mathcal X}C_{D_{\rm ref}}\sigma_{\max}^d C_{\det J}^{|\boldsymbol m|}\frac{((|\boldsymbol m|+d^2-1)!)^\beta}{((d^2-1)!)^\beta}\\
&\quad +\sigma_{\max}^dC_{D_{\rm ref}}^2\|\partial_{\bsy}^{\bsnu-\boldsymbol m}\widehat u(\cdot,\cdot,\bsy)\|_{\mathcal X}C_{A,1}C_{A,2}^{|\boldsymbol m|}\max\{1,2^{|\boldsymbol m|-1}\}\frac{((|\boldsymbol m|+d^2)!)^\beta}{(d^2!)^\beta}\\
&\quad +\|\partial_{\bsy}^{\bsnu-\boldsymbol m}\widehat u(\cdot,\cdot,\bsy)\|_{\mathcal X}C_{A,1}C_{A,2}^{|\boldsymbol m|}\max\{1,2^{|\boldsymbol m|-1}\}\frac{((|\boldsymbol m|+d^2)!)^\beta}{(d^2!)^\beta}\\
&\quad +M^2\|\partial_{\bsy}^{\bsnu-\boldsymbol m}\widehat u(\cdot,\cdot,\bsy)\|_{\mathcal X}\sigma_{\max}^d C_{\det J}^{|\boldsymbol m|}\frac{((|\boldsymbol m|+d^2-1)!)^\beta}{((d^2-1)!)^\beta}\bigg]\\
&\quad +\|\partial_{\bsy}^{\bsnu}\widehat u(\cdot,\cdot,\bsy)\|_{\mathcal X} C_{\Delta,\max}\sigma_{\max}^dC_{f_{\rm ref},1}C_{f_{\rm ref},2}^{|\bsnu|}\max\{1,\|\boldsymbol\rho\|_{\ell^{1/\beta}}^{|\bsnu|}\}\frac{((|\bsnu|+d^2)!)^\beta}{(d^2!)^\beta}\boldsymbol b^{\bsnu}\\&\quad +C_{D_{\rm ref}}\|\partial_{\bsy}^{\bsnu}\widehat u(\cdot,\cdot,\bsy)\|_{\mathcal X} C_{f_{\rm ref},1}C_{f_{\rm ref},2}^{|\bsnu|}\max\{1,\|\boldsymbol\rho\|_{\ell^{1/\beta}}^{|\bsnu|}\}\frac{((|\bsnu|+d^2)!)^\beta}{(d^2!)^\beta}\boldsymbol b^{\bsnu}\\
&\quad + \sum_{\boldsymbol m\leq \bsnu}\bigg[\binom{\bsnu}{\boldsymbol m}M\|\partial_{\bsy}^{\bsnu-\boldsymbol m}\widehat u_0\|_{L^2(D_{\rm ref})}\|\partial_{\bsy}^{\bsnu}\widehat u(\cdot,\cdot,\bsy)\|_{\mathcal X}  \sigma_{\max}^d C_{\det J}^{|\boldsymbol m|}\\
&\quad\quad\quad\quad \times\frac{((|\boldsymbol m|+d^2-1)!)^\beta}{((d^2-1)!)^\beta}\boldsymbol b^{\boldsymbol m}\bigg],
\end{align*}
where we applied the inequalities~\eqref{eq:first}--\eqref{eq:third} as well as the Poincar\'e inequality multiple times.

\paragraph{Step 5.} We derive the following bound on the pullback initial condition.

\begin{lemma}\label{mondaylemma}
Under Assumptions~\ref{a1}\,--\,\ref{assump:gevrey} and~\ref{a10}, there holds
$$
\|\partial_{\bsy}^{\bsnu}\widehat u_0(\cdot,\bsy)\|_{L^2(D_{\rm ref})}\leq 2^{\beta|\bsnu|-\beta}d^{\beta|\bsnu|}C_{u_0}C^{|\bsnu|}\boldsymbol b^{\bsnu}(|\bsnu|!)^{\beta}
$$
for all $\bsnu\in\mathscr F\setminus\{\mathbf 0\}$ and $\bsy\in U$.
\end{lemma}
\begin{proof} We use Fa\`a di Bruno's formula~\citep{savits} to obtain
$$
\partial_{\bsy}^{\bsnu}\widehat u_0(\bsx,\bsy)=\sum_{\substack{\boldsymbol\lambda\in\mathbb N_0^d\\ 1\leq|\boldsymbol\lambda|\leq |\bsnu|}}\partial_{\bsx}^{\boldsymbol\lambda}u_0(\bsx)\bigg|_{\bsx=\boldsymbol V(\bsx,\bsy)}\alpha_{\bsnu,\boldsymbol\lambda}(\bsx,\bsy),
$$
where the coefficients $(\alpha_{\bsnu,\boldsymbol\lambda})$ are defined exactly as in Lemma~\ref{8}. In particular,
\begin{align*}
\|\alpha_{\bsnu,\boldsymbol\lambda}(\cdot,\bsy)\|_{L^\infty}\leq C^{|\bsnu|}\frac{(|\bsnu|!)^{\beta}((|\bsnu|-1)!)^{\beta}}{(\boldsymbol\lambda!)^\beta ((|\bsnu|-|\boldsymbol\lambda|)!)^{\beta}((|\boldsymbol\lambda|-1)!)^{\beta}}\boldsymbol b^{\bsnu}.
\end{align*}
Together with Assumption~\ref{a10}, this yields
\begin{align*}
\|\partial_{\bsy}^{\bsnu}\widehat u_0(\cdot,\bsy)\|_{L^2(D_{\rm ref})}\leq C_{u_0}C^{|\bsnu|}\boldsymbol b^{\bsnu}\sum_{\substack{\boldsymbol\lambda\in\mathbb N_0^d\\ 1\leq |\boldsymbol \lambda|\leq |\bsnu|}}\frac{(|\bsnu|!)^{\beta}((|\bsnu|-1)!)^{\beta}}{(\boldsymbol\lambda!)^{\beta}((|\bsnu|-|\boldsymbol\lambda|)!)^{\beta}((|\boldsymbol\lambda|-1)!)^{\beta}}.
\end{align*}
The claim follows in complete analogy to the proof of Lemma~\ref{8} in the special case $\boldsymbol\rho=\mathbf 1\in\mathbb R^d$.\end{proof}

\paragraph{Step 6.} Putting the lower bound together with the upper bound yields the recurrence relation
\begin{align*}
\|\widehat u(\cdot,\cdot,\bsy)\|_{\mathcal X}&\leq C_0\\
\|\partial_{\bsy}^{\bsnu}\widehat u(\cdot,\cdot,\bsy)\|_{\mathcal X}&\leq \widetilde C_1\sum_{\mathbf 0\neq \boldsymbol m\leq \bsnu}\binom{\bsnu}{\boldsymbol m}\widetilde C_2^{|\boldsymbol m|}\|\partial_{\bsy}^{\bsnu-\boldsymbol m}\widehat u(\cdot,\cdot,\bsy)\|_{\mathcal X}((|\boldsymbol m|+d^2)!)^{\beta}\boldsymbol b^{\boldsymbol m}\\
&\quad + \widetilde C_1\widetilde C_2^{|\bsnu|}((|\bsnu|+d^2)!)^\beta \boldsymbol b^{\bsnu}\\
&\quad + \widetilde C_1 \sum_{\boldsymbol m\leq \bsnu}\binom{\bsnu}{\boldsymbol m}\|\partial_{\bsy}^{\bsnu-\boldsymbol m}\widehat u_0\|_{L^2(D_{\rm ref})}\widetilde C_2^{|\boldsymbol m|}((|\boldsymbol m|+d^2)!)^\beta \boldsymbol b^{\boldsymbol m},
\end{align*}
for all $\bsnu\in\mathscr F\setminus\{\mathbf 0\}$, where \begin{align*}
&C_0=\widetilde C_1(1+\|\widehat u_0\|_{L^2}),\\
&\widetilde C_1=\big(C_{\Delta,\max}\sigma_{\max}^{2d}+C_{D_{\rm ref}}\sigma_{\max}^d+\sigma_{\max}^dC_{D_{\rm ref}}^2C_{A,1}+C_{A,1}+M^2\sigma_{\max}^d\\
&\quad\quad+C_{\Delta,\max}\sigma_{\max}^dC_{f_{\rm ref},1}+C_{D_{\rm ref}}C_{f_{\rm ref},1}+M\sigma_{\max}^d\big)\min\big\{\frac{\sigma_{\min}^{2d}}{C_{D_{\rm ref}}^2C_{\boldsymbol\Delta}^2},\frac{\sigma_{\min}^d}{\sigma_{\max}^2}\big\}^{-1},\\
&\widetilde C_2=4C_{\det J}+4C_{A,2}+2C_{f_{\rm ref,2}}\max\{1,\|\boldsymbol\rho\|_{\ell^{1/\beta}}\}.\end{align*} 

By Lemma~\ref{mondaylemma}, we can simplify the third term in the recurrence bound as follows:
\begin{align*}
&\sum_{\boldsymbol m\leq \bsnu}\binom{\bsnu}{\boldsymbol m}\|\partial_{\bsy}^{\bsnu-\boldsymbol m}\widehat u_0\|_{L^2(D_{\rm ref})}\widetilde C_2^{|\boldsymbol m|}((|\boldsymbol m|+d^2)!)^\beta \boldsymbol b^{\boldsymbol m}\\
&\leq 2^{\beta|\bsnu|-\beta}d^{\beta|\bsnu|}C_{u_0}C^{|\bsnu|} \widetilde C_2^{|\boldsymbol \nu|}\boldsymbol b^{\bsnu}\sum_{\boldsymbol m\leq \bsnu}\binom{\bsnu}{\boldsymbol m}((|\bsnu|-|\boldsymbol m|)!)^{\beta}((|\boldsymbol m|+d^2)!)^{\beta}.
\end{align*}
Here, we can further simplify
\begin{align*}
&\sum_{\boldsymbol m\leq \bsnu}\binom{\bsnu}{\boldsymbol m}((|\bsnu|-|\boldsymbol m|)!)^{\beta}((|\boldsymbol m|+d^2)!)^{\beta}\\&=\sum_{\ell=0}^{|\bsnu|}((|\bsnu|-\ell)!)^{\beta}((\ell+d^2)!)^{\beta}\sum_{\substack{|\boldsymbol m|=\ell\\ \boldsymbol m\leq \bsnu}}\binom{\bsnu}{\boldsymbol m}\\
&=\sum_{\ell=0}^{|\bsnu|}((|\bsnu|-\ell)!)^{\beta}((\ell+d^2)!)^{\beta}\binom{|\bsnu|}{\ell}\\
&\leq \sum_{\ell=0}^{|\bsnu|}((|\bsnu|-\ell)!)^{\beta}((\ell+d^2)!)^{\beta}\binom{|\bsnu|}{\ell}^{\beta}\\
&=(|\bsnu|!)^{\beta}\sum_{\ell=0}^{|\bsnu|}\frac{((\ell+d^2)!)^{\beta}}{(\ell!)^{\beta}}\\
&\leq (|\bsnu|!)^{\beta}\bigg(\sum_{\ell=0}^{|\bsnu|}\frac{(\ell+d^2)!}{\ell!}\bigg)^{\beta}\\
&=(|\bsnu|!)^{\beta}\bigg(\frac{(|\bsnu|+d^2+1)!}{(d^2+1)!|\bsnu|!}\bigg)^{\beta}=\frac{((|\bsnu|+d^2+1)!)^{\beta}}{((d^2+1)!)^{\beta}}.
\end{align*}
This means that the recurrence can be expressed as
\begin{align*}
\|\widehat u(\cdot,\cdot,\bsy)\|_{\mathcal X}&\leq C_0\\
\|\partial_{\bsy}^{\bsnu}\widehat u(\cdot,\cdot,\bsy)\|_{\mathcal X}&\leq C_1\sum_{\mathbf 0\neq \boldsymbol m\leq \bsnu}\binom{\bsnu}{\boldsymbol m}C_2^{|\boldsymbol m|}\|\partial_{\bsy}^{\bsnu-\boldsymbol m}\widehat u(\cdot,\cdot,\bsy)\|_{\mathcal X}((|\boldsymbol m|+d^2)!)^{\beta}\boldsymbol b^{\boldsymbol m}\\
&\quad + C_1C_2^{|\bsnu|}((|\bsnu|+d^2+1
)!)^\beta \boldsymbol b^{\bsnu},
\end{align*}
with $C_1=\widetilde C_1+C_{u_0}$ and $C_2=\widetilde C_2+2^{\beta}d^{\beta}C\widetilde C_2$. We can now apply Lemma~\ref{superlemma} to obtain the inductive bound.

\begin{theorem}\label{parabolicregularitybound}
Let $\bsnu\in\mathscr F$ and $\bsy\in U$. Under Assumptions~\ref{a1}\,--\,\ref{a4} as well as Assumptions~\ref{a10}\,--\,\ref{timederivL2}, the solution to~\eqref{spacetimepullback} satisfies
$$
\|\partial_{\bsy}^{\bsnu}\widehat u(\cdot,\cdot,\bsy)\|_{\mathcal X}\leq C_{\widehat u,1}C_{\widehat u,2}^{|\bsnu|}(|\bsnu|!)^{\beta}\boldsymbol b^{\bsnu},
$$
where the space $\mathcal X$ is defined in~\eqref{eq:x} and
\begin{align*}
&C_{\widehat u,1}=1+C_0,\\
&C_{\widehat u,2}=\widetilde C^2((d^2+1)!)^{\beta}2^{\beta(d^2+2)+1},\\
&\widetilde C=\max\{\widetilde C_1+C_{u_0},\widetilde C_2+2^{\beta}d^{\beta}C\widetilde C_2\},\\
&C_0=\widetilde C_1(1+C_{u_0}).
\end{align*}
\end{theorem}

\section{Error analysis}\label{sec:error}
For the error analysis of the numerical treatment of the considered problems \eqref{eq:weakform} and \eqref{spacetimepullback} one needs to take three types of errors into account. The first one comes from the dimension truncation of the input random field to finitely many terms. The second error is caused by QMC cubature and the last one is the finite element discretization error.  Combining these estimates will yield an overall error estimate. Throughout this section the mentioned errors will be stated for both, the stationary and the parabolic problem. We will measure the error for the stationary case in $\|\cdot\|_{H_0^1(D_{\rm ref})}$ and in the parabolic problem the error will be measured in the norm $\|\cdot\|_{L^2(I;H_0^1(D_{\rm ref}))}$.
\subsection{Truncation error}
In the case of the dimension truncation error one wants to find an upper bound of the expression
\begin{align*}
    \bigg\|\int_{U}(\widehat u(\bsy)-\widehat u_s(\bsy))\,\text{d}\bsy\bigg\|_Z,
\end{align*}
where $Z=H_0^1(D_{\rm ref})$ if $\widehat u$ is the solution of the stationary problem \eqref{eq:weakform} and $Z=L^2(I;H_0^1(D_{\rm ref}))$ if we consider $\widehat u$ to be the solution to the space-time formulation of the parabolic equation \eqref{spacetimepullback}. Here $\widehat u_s$ denotes the dimensionally truncated PDE solution, which is defined as
\begin{align*}
    \widehat u_s(\bsy):=\widehat u( y_1,\dots,y_s,0,0,\dots)\enspace \text{for}\enspace \bsy\in U,\enspace s\in\mathbb{N}.
\end{align*}
The dimension truncation rate follows as an immediate consequence of Theorems~\ref{RegBoundPDEsolStat} and~\ref{parabolicregularitybound} together with~\citet[Theorem~4.3]{DimTruncAnalysisLogNormal}.

\begin{theorem}\label{thm:4}
Under Assumptions~\ref{a1}\,--\,\ref{timederivL2}, there exists a constant $C_{\rm dim}>0$ independent  of $s$ such that
\begin{align}\label{ineq:dm}
    \bigg\|\int_{U}(\widehat u(\bsy)-\widehat u_s(\bsy))\,\textup{d}\bsy\bigg\|_Z\leq C_{\rm dim}s^{-\frac{2}{p}+1},
\end{align}
where $p\in(0,1)$ as in Assumption \ref{a6} and we set $Z=H_0^1(D_{\rm ref})$ if $\widehat u$ is the solution to the pullback Poisson equation. If $\widehat u$ is the solution to the pullback heat equation, then under the additional Assumption~\ref{a10} inequality~\eqref{ineq:dm} holds with $Z=L^2(I;H_0^1(D_{\rm ref}))$.
\end{theorem}

\subsection{QMC error}
We want to quantify the error resulting from the approximation of an integral over the $s$-dimensional unit cube using QMC cubature, which was introduced in Section~\ref{sec:qmc}.

The following theorem gives a choice of POD weights as well as a QMC error bound in the sense of the root-mean square error independent of the dimension $s$.
\begin{theorem}\label{thm:qmcweight}
Let $\widehat u_s$ denote either the solution to the Poisson problem, in which case we assume that $Z=H_0^1(D_{\rm ref})$ and Assumptions~\ref{a1}\,--\,\ref{a6} hold, or let $\widehat u_s$ denote the solution to the heat equation, in which case we assume that $Z=L^2(I;H_0^1(D_{\rm ref}))$ and Assumptions~\ref{a1}\,--\,\ref{a6} and~\ref{a10}\,--\,\ref{timederivL2}. Then the root-mean-square error using a randomly shifted lattice rule \eqref{QMC_rule} with $n=2^m$, $m \in \mathbb{N}$, points  constructed by a CBC algorithm with $R$ independent random shifts  satisfies
\begin{align*}
\sqrt{\mathbb E_{\boldsymbol\Delta}\|I_s(\widehat u_s)-Q_{\rm ran}(\widehat u_s)\|_Z^2}\leq C\cdot n^{-\min\{\frac{1}{p}-\frac12 ,1 -\varepsilon\}}\quad\text{for arbitrary $\varepsilon\in(0,\tfrac12)$},
\end{align*}
where $C>0$ is independent of $s$ 
when the weights $(\gamma_{\setu})_{\setu\subset\mathbb N}$ are chosen as the sequence of product and order dependent (POD) weights
\begin{align*}
&\gamma_{\mathfrak u}=\bigg((|\setu|!)^{\beta}\prod_{j\in\setu}\frac{C_{\widehat u,2}b_j}{\sqrt{2\zeta(2\lambda)/(2\pi^2)^{\lambda}}}\bigg)^{\frac{2}{1+\lambda}},\\
&\lambda=\begin{cases}
\frac{p}{2-p}&\text{if}~p\in(\tfrac23,\tfrac{1}{\beta}),\\
\frac{1}{2-2\varepsilon}&\text{if}~p\in(0,\min\{\tfrac23,\frac{1}{\beta}\}],~p\neq \frac{1}{\beta},
\end{cases}
\end{align*}
where $C_{\widehat u,2}$ is defined in Theorems~\ref{RegBoundPDEsolStat}--\ref{parabolicregularitybound}.%
\end{theorem}
\begin{proof} The proof is carried out in complete analogy to~\cite{kss12}.\end{proof}

\subsection{Finite element error}
Let Assumptions \ref{a1}\,--\,\ref{a6} and \ref{a9} hold. For the finite element discretization with respect to the spatial variable $\bsx\in D_{\rm ref}$, we consider  finite element subspaces of $H_0^1(D_\text{ref})$, denoted by $\{V_h\}_{h}$, where $h$ refers to the mesh size of a regular triangulation of $D_{\rm ref}$. As the basis functions we choose continuous piecewise linear functions spanning the finite element spaces. With respect to the temporal variable $t\in I$ in the pullback heat equation, we shall employ an implicit Euler time discretization with time step size $\Delta t>0$.
 For further details, we refer to~\cite{fembook}.

\begin{theorem}\label{thm:5}
\begin{list}{}{\setlength{\leftmargin}{1.3\parindent}\setlength{\labelwidth}{1.3\parindent}\setlength{\labelsep}{0.5em}}
\item[]
\item[\rm (i)] Let $\widehat u_{s,h}(\cdot,\bsy)$, $\bsy\in U$, denote the dimensionally truncated finite element solution to the pullback Poisson equation with finite element mesh width $h>0$. Under Assumptions~\ref{a1}\,--\,\ref{a6} and \ref{a9}, there holds
$$
\|\widehat u_s(\cdot,\bsy)-\widehat u_{s,h}(\cdot,\bsy)\|_{L^2(D_{\rm ref})}\leq C_{\rm FEM}h^2,
$$
where the constant $C_{\rm FEM}>0$ is independent of $s$, $h$, and $\bsy$.
\item[\rm (ii)] Let $\widehat u_{s,h,\Delta t}(\cdot,\cdot,\bsy)$, $\bsy\in U$, denote the dimensionally truncated space-time discretized solution to the pullback heat equation obtained using an implicit Euler time discretization with time step size $\Delta t>0$ and finite element mesh width $h>0$. Under Assumptions~\ref{a1}\,--\,\ref{a6} and \ref{a9}\,--\,\ref{timederivL2}, there holds
$$
\|\widehat u_s(\cdot,\cdot,\bsy)-\widehat u_{s,h,\Delta t}(\cdot,\cdot,\bsy)\|_{L^2(I;L^2(D_{\rm ref}))}\leq C_{\rm FEM}(\Delta t+h^2),
$$
where the constant $C_{\rm FEM}>0$ is independent of $s$, $h$, $\Delta t$, and $\bsy$.
\end{list}
\end{theorem}
\begin{proof} In the stationary case, we refer to~\cite{HHKKS24}. For the time-dependent case, we refer the reader to~\cite{fempaper}.\end{proof}

\subsection{Overall error}
The combination of the aforementioned error estimates, i.e., truncation error, QMC error and finite element error yields an overall error bound for both considered problems.
\Needspace{5\baselineskip}
\begin{theorem}
\begin{list}{}{\setlength{\leftmargin}{1.3\parindent}\setlength{\labelwidth}{1.3\parindent}\setlength{\labelsep}{0.5em}}
\item[]
\item[\rm (i)] Let $\widehat u$ be the solution of the  pullback variational formulation of the Poisson equation~\eqref{eq:weakform}  and let $\widehat u_{s,h}$ denote the dimensionally truncated finite element solution. Under the Assumptions~\ref{a1}\,--\,\ref{a9} we find in the stationary setting
\begin{align*}
    \sqrt{\mathbb{E}_\Delta\|I_s(\widehat u)-Q_{\rm ran}(\widehat u_{s,h})\|_{L^2(D_{\rm ref})}^2}\leq C_{\rm all}\big(s^{-\frac{2}{p}+1}+h^2+n^{-\min\left\{\frac{1}{p}-\frac{1}{2},1-\varepsilon\right\}}\big)
\end{align*}
with $s$ denoting the truncation index, $h$ being the the mesh size used in the finite element discretization and $C_{\rm all}>0$ being a positive constant, which is independent of $s,h$ and the number of lattice points $n$.
\item[\rm (ii)]Let $\widehat u$ denote the solution to the pullback weak formulation of the heat equation~\eqref{spacetimepullback}and let $\widehat u_{s,h,\Delta t}$ denote the dimensionally truncated space-time finite element solution. For the parabolic case one obtains under Assumptions~\ref{a1}\,--\,\ref{timederivL2} the overall error bound
\begin{align*}
    &\sqrt{\mathbb{E}_\Delta\|I_s(\widehat u)-Q_{\rm ran}(\widehat u_{s,h,\Delta t})\|_{L^2(I;L^2(D_{\rm ref}))}^2}\\
    &\leq C_{\rm all}\left(s^{-\frac{2}{p}+1}+h^2+n^{-\min\left\{\frac{1}{p}-\frac{1}{2},1-\varepsilon\right\}}+\Delta t\right),  
\end{align*}
where $\Delta t$ denotes the step size of the implicit Euler scheme for the discretization of the temporal variable and $C_{\rm all}>0$ is a constant, which does not depend on $s,h,n$ and $\Delta t$.
\end{list}\end{theorem}
\begin{proof}
The claim follows by combining Theorems~\ref{thm:4}--\ref{thm:5}.
\end{proof}
\section{Numerical Experiments}\label{sec:numex}
In this section, we consider a series of numerical experiments to approximate the expected values of solutions to the Poisson equation and the heat equation subject to domain uncertainty using QMC. For the numerical experiments, we fix the reference domain
$$
D_{\rm ref}=\{\bsx\in\mathbb R^2\mid x_1^2+x_2^2\leq 1\}\subset\mathbb R^2
$$
and consider domains parameterized by $D(\bsy)=\boldsymbol V(D_{\rm ref},\bsy)$ for $\bsy\in U_s$, where
\begin{align*}
&\boldsymbol V(\bsx,\bsy)=\bigg[1+\xi\bigg(\sum_{j=1}^s y_j\psi_j(\bsx)\bigg)\bigg]\bsx,\\
&\psi_j(\bsx)=j^{-\vartheta}\sin\big(3j({\rm atan2}(x_1,x_2)+\pi)\big),\quad j\geq 1,
\end{align*}
for $\bsx\in D_{\rm ref}$ and $\bsy\in U_s$.  We set $s=100$ as the truncation dimension and consider the following four experiments:
\begin{list}{{\rm (E\arabic{exampleenum})}}{\usecounter{exampleenum}%
  \setlength{\leftmargin}{1.7\parindent}%
  \setlength{\labelwidth}{1.7\parindent}%
  \setlength{\labelsep}{0.5em}%
  \renewcommand{\makelabel}[1]{\hfill #1}}
\item $\xi(y)=y$ and $\vartheta=2.1$;\label{eq:ex1}
\item $\xi(y)=\exp\big(-\frac{1}{y+\zeta(\vartheta)/2}\big)$ and $\vartheta=2.1$;\label{eq:ex2}
\item $\xi(y)=y$ and $\vartheta=2.5$;\label{eq:ex3}
\item $\xi(y)=\exp\big(-\frac{1}{y+\zeta(\vartheta)/2}\big)$ and $\vartheta=2.5$.\label{eq:ex4}
\end{list}
Experiments~\ref{eq:ex1} and~\ref{eq:ex3} correspond to vectorized Karhunen--Lo\`eve expansions and are thereby covered not only by our theory, but also by the existing results presented in~\cite{HPS16}. Meanwhile, experiments~\ref{eq:ex2} and~\ref{eq:ex4} do not produce domain mappings which are holomorphic with respect to the parametric variable $\bsy$, but they are covered by the theory developed for Gevrey regular domain mappings in the present work. %
Some realizations of the parametric domains corresponding to experiments~\ref{eq:ex1} and~\ref{eq:ex2} are illustrated in Figure~\ref{fig:realiz} for randomly generated $\bsy\in U_s$ alongside the finite element discretization used for the reference domain $D_{\rm ref}$. The same finite element mesh is used for all experiments. In particular, we solve the pullback weak formulation of the Poisson equation and the heat equation using a first-order finite element mesh with mesh width $h=0.1$. Meanwhile, the pullback heat equation is solved by first discretizing the temporal variable using the implicit Euler method with time step size $\Delta t=0.1$. The heat equation was solved over the time interval $I=[0,1]$ subject to the initial heat distribution $u_0\equiv 1$.
\begin{figure}[!t]
\centering
\subfloat{
\begin{tikzpicture}
\node (img) {\includegraphics[height=.275\textwidth]{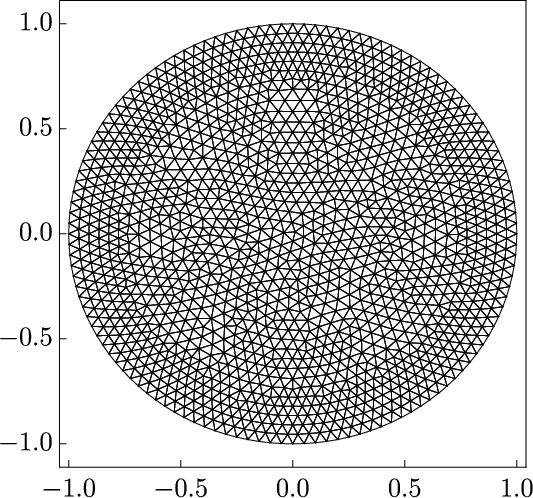}};
\node[below=of img,node distance=0cm,xshift=.2cm,yshift=1cm]{$x_1$};
\node[left=of img,node distance=0cm,rotate=90,anchor=center,xshift=.2cm,yshift=-1.0cm]{$x_2$};
\end{tikzpicture}}\hspace*{-.2cm}
\subfloat{
\begin{tikzpicture}
\node (img) {\includegraphics[height=.275\textwidth]{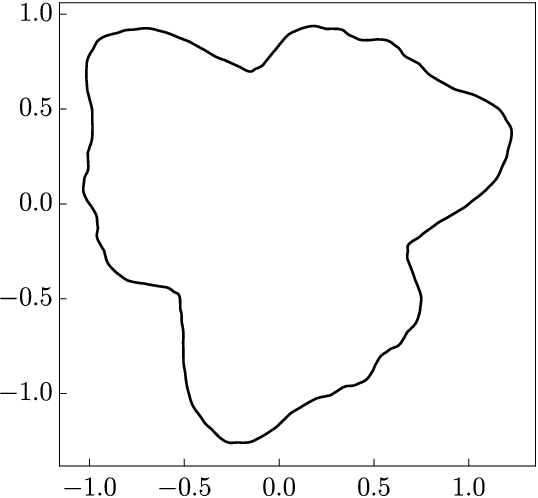}};
\node[below=of img,node distance=0cm,xshift=.2cm,yshift=1cm]{$x_1$};
\node[left=of img,node distance=0cm,rotate=90,anchor=center,xshift=.2cm,yshift=-1.0cm]{$x_2$};
\end{tikzpicture}}\hspace*{-.2cm}
\subfloat{
\begin{tikzpicture}
\node (img) {\includegraphics[height=.275\textwidth]{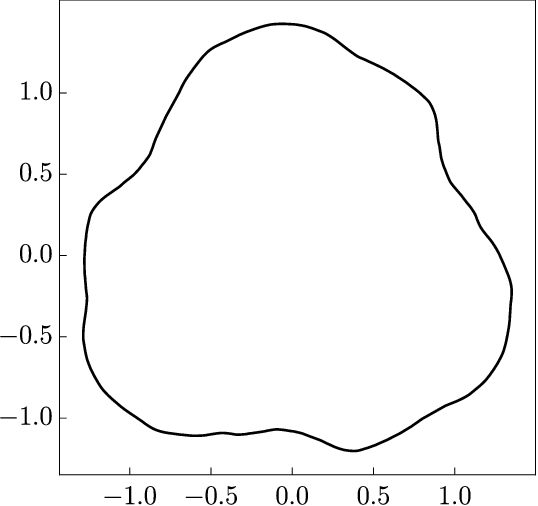}};
\node[below=of img,node distance=0cm,xshift=.2cm,yshift=1cm]{$x_1$};
\node[left=of img,node distance=0cm,rotate=90,anchor=center,xshift=.2cm,yshift=-1.0cm]{$x_2$};
\end{tikzpicture}}
\caption{Left: the triangulation of the reference domain $D_{\rm ref}$ used in all numerical experiments. Middle: a random realization of the random field corresponding to experiment~\ref{eq:ex1}. Right: a random realization of the random field corresponding to experiment~\ref{eq:ex2}.}\label{fig:realiz}
\end{figure}
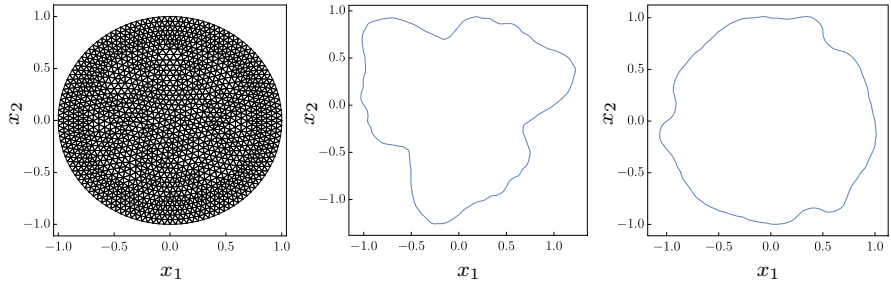

\begin{figure}[!t]
\centering
\subfloat{
\begin{tikzpicture}
\node (img) {\includegraphics[height=.35\textwidth]{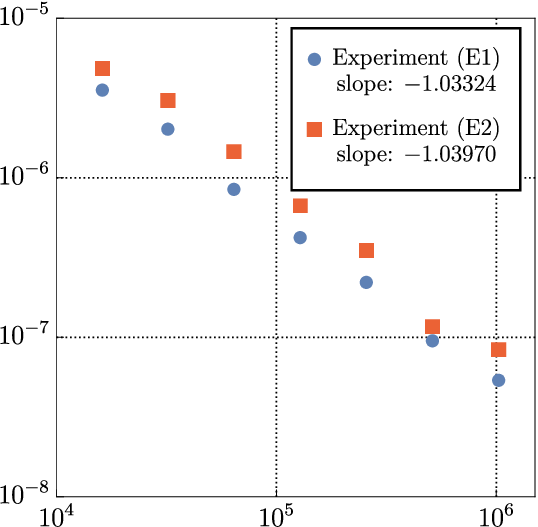}};
\node[below=of img,node distance=0cm,xshift=.2cm,yshift=1cm]{number of function evaluations $nR$};
\node[left=of img,node distance=0cm,rotate=90,anchor=center,xshift=.2cm,yshift=-.7cm]{R.M.S.~error};
\end{tikzpicture}}\hspace*{.2cm}
\subfloat{
\begin{tikzpicture}
\node (img) {\includegraphics[height=.35\textwidth]{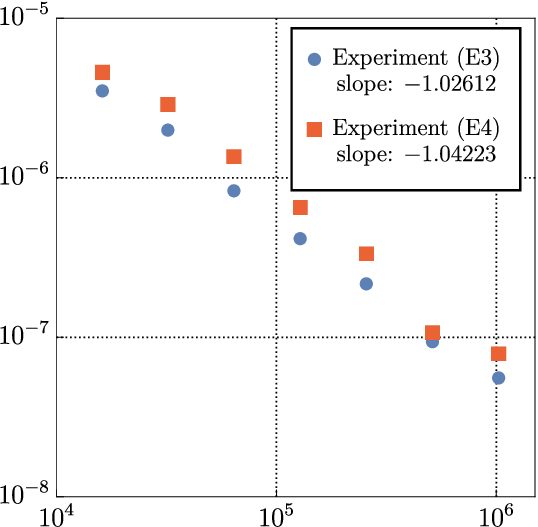}};
\node[below=of img,node distance=0cm,xshift=.2cm,yshift=1cm]{number of function evaluations $nR$};
\node[left=of img,node distance=0cm,rotate=90,anchor=center,xshift=.2cm,yshift=-.7cm]{R.M.S.~error};
\end{tikzpicture}}
\caption{The computed $L^2(D_{\rm ref})$ errors for the Poisson equation. Left: the numerical cubature errors corresponding to experiments~\ref{eq:ex1}--\ref{eq:ex2}. Right: the numerical cubature errors corresponding to experiments~\ref{eq:ex3}--\ref{eq:ex4}.}\label{fig:results}
%
\subfloat{
\begin{tikzpicture}
\node (img) {\includegraphics[height=.35\textwidth]{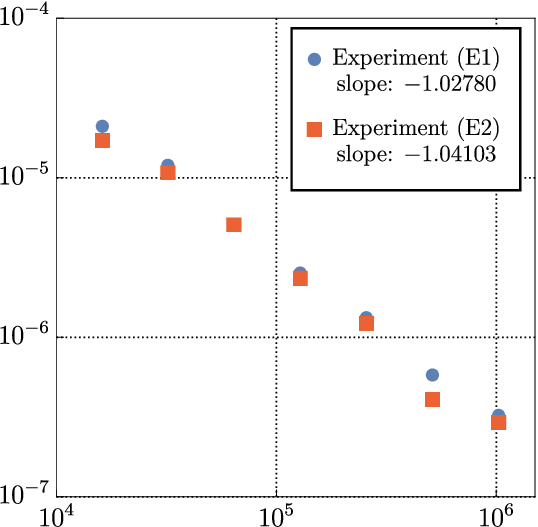}};
\node[below=of img,node distance=0cm,xshift=.2cm,yshift=1cm]{number of function evaluations $nR$};
\node[left=of img,node distance=0cm,rotate=90,anchor=center,xshift=.2cm,yshift=-.7cm]{R.M.S.~error};
\end{tikzpicture}}\hspace*{.2cm}
\subfloat{
\begin{tikzpicture}
\node (img) {\includegraphics[height=.35\textwidth]{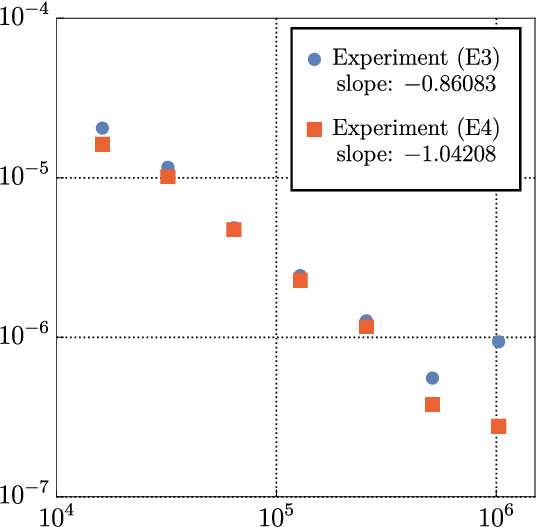}};
\node[below=of img,node distance=0cm,xshift=.2cm,yshift=1cm]{number of function evaluations $nR$};
\node[left=of img,node distance=0cm,rotate=90,anchor=center,xshift=.2cm,yshift=-.7cm]{R.M.S.~error};
\end{tikzpicture}}
\caption{The computed $H_0^1(D_{\rm ref})$ errors for the Poisson equation. Left: the numerical cubature errors corresponding to experiments~\ref{eq:ex1}--\ref{eq:ex2}. Right: the numerical cubature errors corresponding to experiments~\ref{eq:ex3}--\ref{eq:ex4}.}\label{fig:resultsb}
\end{figure}

It is straightforward to verify that
$$
\|\partial_{\bsy}^{\bsnu}\boldsymbol V(\cdot,\bsy)\|_{W^{1,\infty}(D_{\rm ref})}\lesssim \begin{cases}
|\bsnu|!\boldsymbol b^{\bsnu}&\text{in experiments~\ref{eq:ex1} and~\ref{eq:ex3}},\\
(|\bsnu|!)^2 \boldsymbol b^{\bsnu}&\text{in experiments~\ref{eq:ex2} and~\ref{eq:ex4}},
\end{cases}
$$
where $\boldsymbol b=(b_j)_{j\geq 1}$ is defined by $b_j=\|\psi_j\|_{W^{1,\infty}(D_{\rm ref})}\propto j^{1-\vartheta}$. For numerical stability, we set $C_{\widehat u,2}=1$ in the expression for the QMC weights derived in Theorem~\ref{thm:qmcweight}  and use the QMC4PDE software~\citep{qmc4pde} to obtain the generating vector for $\vartheta\in\{2.1,2.5\}$. For the approximation of $\mathbb E[\widehat u]$, we use $R=16$ random shifts. The root-mean-square (R.M.S.) errors are approximated using the formula
$$
\sqrt{\frac{1}{R(R-1)}\sum_{r=1}^R \|Q_{\rm ran}(\widehat u)-Q_r(\widehat u)\|_{\mathcal H}^2},
$$
where we used both $\mathcal H=H_0^1(D_{\rm ref})$ and $\mathcal H=L^2(D_{\rm ref})$ for the solution to the Poisson equation and $\mathcal H=\mathcal L^2$ and $\mathcal H=L^2(I;H_0^1(D_{\rm ref}))$ for the heat equation. Since the Poincar\'e inequality can be used to bound $\|\cdot\|_{L^2}$ and $\|\cdot\|_{\mathcal L^2}$ by $\|\cdot\|_{H_0^1}$ and $\|\cdot\|_{L^2(I;H_0^1(D_{\rm ref}))}$, respectively, we expect the same rates of QMC convergence in all experiments.
The calculated R.M.S.~errors are displayed in Figures~\ref{fig:results}--\ref{fig:resultsb} for the Poisson problem and in Figures~\ref{fig:results2}--\ref{fig:results2b} for the heat equation corresponding to the two different norms. In all cases, we observe that the R.M.S.~error decays at an essentially linear rate $\mathcal O(n^{-1})$, as proved in Theorem~\ref{thm:qmcweight}. However, since the $L^2$ and $\mathcal L^2$ norms are weaker, the magnitudes of the computed errors are smaller than the errors measured using Sobolev norms. In addition, we note that the $H_0^1$ norm appears to be sensitive to numerical discretization errors, resulting in a loss of accuracy in the R.M.S.~error below a certain threshold as displayed in Figure~\ref{fig:resultsb}.

\begin{figure}[!t]
\centering
\subfloat{
\begin{tikzpicture}
\node (img) {\includegraphics[height=.35\textwidth]{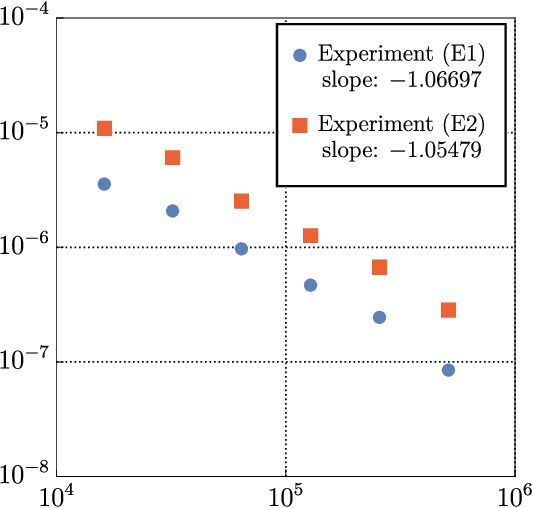}};
\node[below=of img,node distance=0cm,xshift=.2cm,yshift=1cm]{number of function evaluations $nR$};
\node[left=of img,node distance=0cm,rotate=90,anchor=center,xshift=.2cm,yshift=-.7cm]{R.M.S.~error};
\end{tikzpicture}}\hspace*{.2cm}
\subfloat{
\begin{tikzpicture}
\node (img) {\includegraphics[height=.35\textwidth]{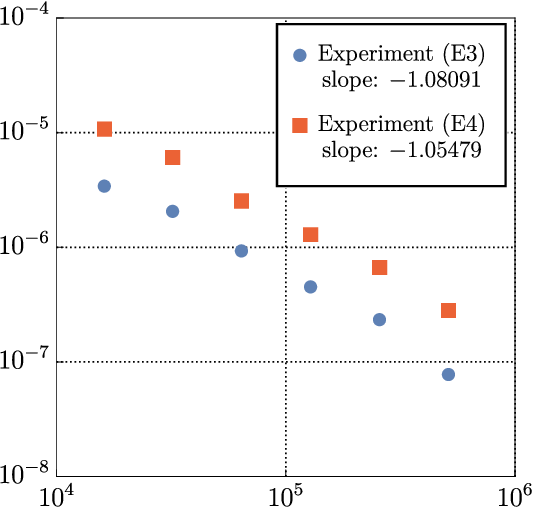}};
\node[below=of img,node distance=0cm,xshift=.2cm,yshift=1cm]{number of function evaluations $nR$};
\node[left=of img,node distance=0cm,rotate=90,anchor=center,xshift=.2cm,yshift=-.7cm]{R.M.S.~error};
\end{tikzpicture}}
\caption{The computed  $L^2(I;L^2(D_{\rm ref}))$ errors for the heat equation. Left: the numerical cubature errors corresponding to experiments~\ref{eq:ex1}--\ref{eq:ex2}. Right: the numerical cubature errors corresponding to experiments~\ref{eq:ex3}--\ref{eq:ex4}.}\label{fig:results2}
%
\subfloat{
\begin{tikzpicture}
\node (img) {\includegraphics[height=.35\textwidth]{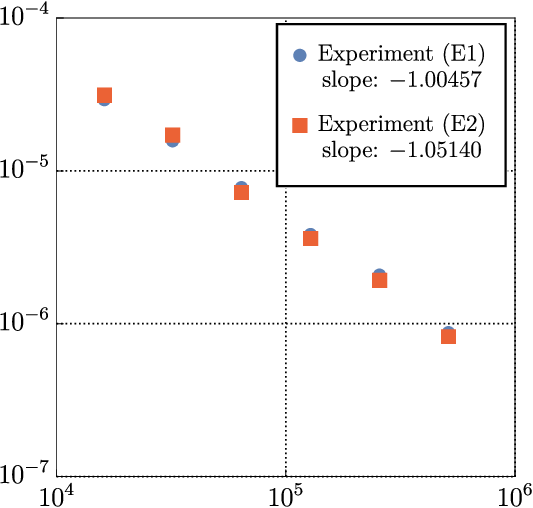}};
\node[below=of img,node distance=0cm,xshift=.2cm,yshift=1cm]{number of function evaluations $nR$};
\node[left=of img,node distance=0cm,rotate=90,anchor=center,xshift=.2cm,yshift=-.7cm]{R.M.S.~error};
\end{tikzpicture}}\hspace*{.2cm}
\subfloat{
\begin{tikzpicture}
\node (img) {\includegraphics[height=.35\textwidth]{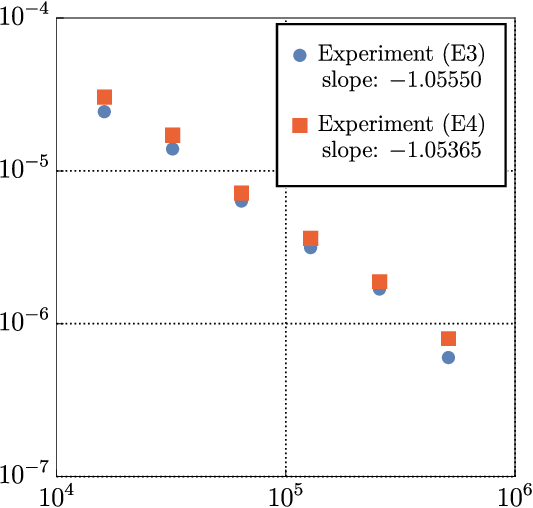}};
\node[below=of img,node distance=0cm,xshift=.2cm,yshift=1cm]{number of function evaluations $nR$};
\node[left=of img,node distance=0cm,rotate=90,anchor=center,xshift=.2cm,yshift=-.7cm]{R.M.S.~error};
\end{tikzpicture}}
\caption{The computed $L^2(I;H_0^1(D_{\rm ref}))$ errors for the heat equation. Left: the numerical cubature errors corresponding to experiments~\ref{eq:ex1}--\ref{eq:ex2}. Right: the numerical cubature errors corresponding to experiments~\ref{eq:ex3}--\ref{eq:ex4}.}\label{fig:results2b}
\end{figure}

\section{Conclusions}\label{sec:conclusions}

In this work, we have investigated uncertainty quantification for random domains subject to the Poisson equation and the heat equation. Specifically, we find that modeling the input random fields as Gevrey regular vector fields leads to faster-than-Monte Carlo convergence rates for the computation of the statistical response when using randomly shifted rank-1 lattice rules. The Gevrey model accommodates more general parameterizations of domain uncertainty: while the previous literature is mainly concerned with holomorphic transformations of a fixed reference domain, the Gevrey class also covers non-holomorphic representations. 

For the construction of the QMC rules, we performed a detailed parametric regularity analysis informing the choice of optimized QMC point sets. In addition to the numerical cubature error, we also analyzed the approximation errors stemming from dimension truncation and finite element discretization.  We remark that our QMC analysis is carried out in a general manner and thus independent of the chosen PDE discretization scheme. Some potential future directions include extending QMC analysis for curved or time dependent domains as well as forward and inverse domain uncertainty quantification governed by more involved elliptic or parabolic PDEs.

\section*{Acknowledgements}
ADj and CS acknowledge support from DFG CRC/TRR 388 ``Rough
Analysis, Stochastic Dynamics and Related Fields'', Project B06. AZ is grateful for travel funding provided by BMS and DAAD making it possible to present some of the produced results. The work of VK was supported by the Research Council of Finland (Flagship of Advanced Mathematics for Sensing, Imaging and Modelling grant 359183). The authors wish to acknowledge CSC -- IT Center for Science, Finland, for computational resources.


\section*{Appendix}\label{sec:appendix}
This section covers results on technical recurrence relations and PDE regularity used in our analysis.

Let
\begin{align}
\tau_{0,\beta,q}=1\quad\text{and}\quad \tau_{k,\beta,q}=\sum_{\ell=0}^{k-1}\frac{((k-\ell+q)!)^\beta}{((k-\ell)!)^\beta}\tau_{\ell,\beta,q},\quad k\geq 1.\label{eq:tauseq}
\end{align}
This sequence is a generalization of sequence A003480 in \emph{The On-Line Encyclopedia of Integer Sequences (OEIS)}.

The sequence~\eqref{eq:tauseq} appears in the parametric regularity analysis developed in Section~\ref{sec:paramreg} via the following abstract recursive bound.

\begin{lemma}\label{lemma:taubnd} Let $(\Upsilon_\bsnu)_{\bsnu\in\mathscr F}$ be a sequence defined by the recurrence relation
$$
\Upsilon_{\mathbf 0}=1\quad\text{and}\quad\Upsilon_{\boldsymbol m}=\sum_{\substack{\boldsymbol w\leq \boldsymbol m\\ \boldsymbol w\neq\mathbf 0}}\binom{\boldsymbol m}{\boldsymbol w}((|\boldsymbol w|+q)!)^{\beta}\Upsilon_{\boldsymbol m-\boldsymbol w},\quad \boldsymbol m\in\mathscr F\setminus\{\mathbf 0\}.
$$
Then there holds
$$
\Upsilon_{\bsnu}\leq \tau_{|\bsnu|,\beta,q}(|\bsnu|!)^\beta\quad\text{for all}~\bsnu\in\mathscr F,
$$
where the sequence $(\tau_{k,\beta,q})_{k\geq 0}$ is defined by~\eqref{eq:tauseq}.
\end{lemma}
\begin{proof} We prove the claim by induction with respect to the order of the multi-index $\bsnu$. The claim is clearly true if $\bsnu=\mathbf 0$. We fix $\bsnu\in\mathscr F\setminus\{\mathbf 0\}$ and assume that the claim holds for all multi-indices with order less than $|\bsnu|$. Then there holds
\begin{align*}
 \Upsilon_\bsnu &=\sum_{\substack{\boldsymbol w\leq \bsnu\\ \boldsymbol w\neq\mathbf 0}}\binom{\bsnu}{\boldsymbol{w}}((|\boldsymbol{w}|+q)!)^\beta\Upsilon_{\bsnu-\boldsymbol{w}}\\
&=\sum_{\substack{\boldsymbol w\leq \bsnu \\ \boldsymbol w\neq\bsnu}}\binom{\bsnu}{\boldsymbol{w}}((|\bsnu|-|\boldsymbol{w}|+q)!)^\beta\Upsilon_{\boldsymbol{w}}\\
&\leq \sum_{\substack{\boldsymbol w\leq \bsnu \\ \boldsymbol w\neq\bsnu}}\binom{\bsnu}{\boldsymbol{w}}((|\bsnu|-|\boldsymbol{w}|+q)!)^\beta \tau_{|\boldsymbol{w}|,\beta,q}(|\boldsymbol{w}|!)^\beta\\
&= \sum_{\ell=0}^{|\bsnu|-1}((|\bsnu|-\ell+q)!)^\beta\tau_{\ell,\beta,q}(\ell!)^\beta\sum_{\substack{\boldsymbol w\leq \bsnu \\ |\boldsymbol w|=\ell}}\binom{\bsnu}{\boldsymbol{w}}\\
&\leq \sum_{\ell=0}^{|\bsnu|-1}((|\bsnu|-\ell+q)!)^\beta\tau_{\ell,\beta,q}(\ell!)^\beta \bigg(\sum_{\substack{\boldsymbol w\leq \bsnu \\ |\boldsymbol w|=\ell}}\binom{\bsnu}{\boldsymbol{w}}\bigg)^\beta\\
&\leq (|\bsnu|!)^\beta \sum_{\ell=0}^{|\bsnu|-1}((|\bsnu|-\ell+q)!)^\beta\tau_{\ell,\beta,q}=(|\bsnu|!)^\beta\tau_{|\bsnu|,\beta,q},
\end{align*}
which proves the statement. \end{proof}

We derive an explicit bound on $\tau_{|\bsnu|,\beta,d}$.

\begin{lemma}\label{lemma:taubound}
There holds
$$\tau_{k,\beta,q}\leq (q!)^{\beta k}2^{\beta (q+1)k}\max\{1,2^{k-1}\}.$$
\end{lemma}
\begin{proof}
We let $q\in\mathbb Z_+$ be arbitrary and prove the claim by induction with respect to $k\in\mathbb Z_+$. The basis of induction $k=0$ follows from
$$
\tau_{0,\beta,q}=1.
$$
Next, we fix $k\geq 1$ and assume that the claim is true for all values up to but not including $k$. Then, by utilizing the fact that $\frac{(k-\ell+q)!}{(k-\ell)!}=q!\binom{k-\ell+q}{q}\leq q!2^{k-\ell+q}$, we obtain that
\begin{align*}
\tau_{k,\beta,q}&=\sum_{\ell=0}^{k-1}\frac{((k-\ell+q)!)^\beta}{((k-\ell)!)^\beta}\tau_{\ell,\beta,q}\\
&\leq \sum_{\ell=0}^{k-1}(q!)^\beta 2^{\beta (k-\ell+q)}(q!)^{\beta \ell} 2^{\beta (q+1)\ell}\max\{1,2^{\ell-1}\}\\
&\leq (q!)^{\beta k} 2^{\beta k+\beta q+\beta q(k-1)}\sum_{\ell=0}^{k-1}\max\{1,2^{\ell-1}\}\\
&=(q!)^{\beta k} 2^{\beta (q+1)k}2^{k-1}
\end{align*}
as desired.
\end{proof}

\begin{lemma}\label{lemma:xi}
Let $(\xi_\bsnu)_{\bsnu\in\mathscr{F}}$ be a sequence satisfying the recurrence relation~\eqref{eq:xirecu1}--\eqref{eq:xirecu2} with $C\geq 1$. Then, there holds
\begin{align*}
    \xi_\bsnu \leq \sigma_{\min}^{-|\bsnu|-1} C^{|\bsnu|}(|\bsnu|!)^\beta\max\{1,2^{|\bsnu|-1}\} \boldsymbol{b}^\bsnu \enspace \text{for all}\enspace \bsnu\in \mathscr{F}.
\end{align*}
\end{lemma}
\begin{proof}
We begin by proving the inequality
\begin{align}\label{eq:thisinequality}
    \xi_\bsnu \leq \sigma_{\min}^{-|\bsnu|-1} C^{|\bsnu|}\alpha_\bsnu \boldsymbol{b}^\bsnu \enspace \text{for all}\enspace \bsnu\in \mathscr{F},
\end{align}
where
\begin{align}
\alpha_{\mathbf 0}=1\quad\text{and}\quad \alpha_\bsnu=\sum_{ \boldsymbol{0}\neq{\boldsymbol{m}}\leq \bsnu}\binom{\bsnu}{\boldsymbol{m}}(|\boldsymbol{m}|!)^\beta \alpha_{\bsnu-\boldsymbol{m}},\quad \bsnu\in\mathscr F\setminus\{\mathbf 0\},\label{eq:alphaseqdef}
\end{align}
by induction over the multi-index $\bsnu$.

The base case $\bsnu=\boldsymbol{0}$ is given by
\begin{align*}
    \xi_{\boldsymbol{0}}=\|\partial_\bsy^{\boldsymbol{0}}J(\cdot,\bsy)^{-1}\|_{L^\infty}\leq \sigma_{\min}^{-1}.
\end{align*}
For the induction hypothesis we fix a multi-index $\bsnu\in\mathscr{F}\setminus\{\boldsymbol{0}\}$ and assume that the claim holds true for all multi-indices with order less than $|\bsnu|$.

One finds
\begin{align*}
    \xi_\bsnu &\leq C\sigma_{\min}^{-1}\sum_{\boldsymbol{0}\neq \boldsymbol{m}\leq \bsnu}\binom{\bsnu}{\boldsymbol{m}}(|\boldsymbol{m}|!)^\beta\boldsymbol{b}^{\boldsymbol{m}} \xi_{\bsnu-\boldsymbol{m}}\\
    &\leq C\sigma_{\min}^{-1}\sum_{\boldsymbol{0}\neq \boldsymbol{m}\leq \bsnu}\binom{\bsnu}{\boldsymbol{m}}(|\boldsymbol{m}|!)^\beta\boldsymbol{b}^{\boldsymbol{m}} \sigma_{\min}^{-|\bsnu|+|\boldsymbol{m}|-1} \max \{1,C\}^{|\bsnu|-|\boldsymbol{m}|}\alpha_{\bsnu-\boldsymbol{m}}\boldsymbol{b}^{\bsnu-\boldsymbol{m}}\\
    &\leq C^{|\bsnu|}\sigma_{\min}^{-|\bsnu|-1}\boldsymbol{b}^{\bsnu}\sum_{\boldsymbol{0}\neq \boldsymbol{m}\leq \bsnu}\binom{\bsnu}{\boldsymbol{m}}(|\boldsymbol{m}|!)^\beta \alpha_{\bsnu-\boldsymbol{m}}\\
    &=  C^{|\bsnu|}\sigma_{\min}^{-|\bsnu|-1}\boldsymbol{b}^{\bsnu}\alpha_\bsnu.
\end{align*}
This proves inequality~\eqref{eq:thisinequality}. 

 Next we show the following bound for the sequence~\eqref{eq:alphaseqdef}: 
\begin{align*}
    \alpha_\bsnu\leq \max\{1,2^{|\bsnu|-1}\}(|\bsnu|!)^\beta.
\end{align*}
We do this by induction. The induction base follows directly from the definition $\alpha_{\mathbf 0} = 1$ and this in fact coincides with $\max\{1,2^{0-1}\}(0!)^\beta =\max\{1,2^{-1}\}=1.$

Now let $\bsnu\in\mathscr F\setminus\{\mathbf 0\}$ and suppose that the claim has already been proven for all multi-indices with order less than $|\bsnu|$. Then, applying the definition of the sequence $\alpha$ and the induction hypothesis, we obtain 
\begin{align*}
\alpha_{\bsnu}&\leq \sum_{\mathbf 0\neq \boldsymbol m\leq \bsnu}\binom{\bsnu}{\boldsymbol m}(|\boldsymbol m|!)^\beta \max\{1,2^{|\bsnu|-|\boldsymbol m|-1}\}((|\bsnu|-|\boldsymbol m|)!)^\beta\\
&=(|\bsnu|!)^\beta+2^{|\bsnu|-1}\sum_{\mathbf 0\neq \boldsymbol m< \bsnu}\binom{\bsnu}{\boldsymbol m}(|\boldsymbol m|!)^\beta 2^{-|\boldsymbol m|}((|\bsnu|-|\boldsymbol m|)!)^\beta\\
&=(|\bsnu|!)^\beta+2^{|\bsnu|-1}\sum_{\ell=1}^{|\bsnu|-1}(\ell!)^\beta 2^{-\ell}((|\bsnu|-\ell)!)^\beta\frac{|\bsnu|!}{\ell!(|\bsnu|-\ell)!}\\
&\leq (|\bsnu|!)^\beta +(|\bsnu|!)^\beta2^{|\bsnu|-1}\underset{=1-2^{1-|\bsnu|}}{\underbrace{\sum_{\ell=1}^{|\bsnu|-1}2^{-\ell}}}\\
&=(|\bsnu|!)^\beta2^{|\bsnu|-1}.
\end{align*}
Note that this is precisely $\alpha_{\bsnu}$ when $\bsnu\neq\mathbf 0$, which completes the induction proof.\end{proof}

The following result is used to resolve the inductive bounds for both the stationary and parabolic settings.

\begin{lemma}\label{superlemma}
Let $(\Lambda_{\bsnu})_{\bsnu\in\mathscr F}$ satisfy, for all $\bsnu\in\mathscr F\setminus\{\mathbf 0\}$,
\begin{align*}
&\Lambda_{\mathbf 0}\leq C_0,\\
&\Lambda_{\bsnu}\leq C\sum_{\substack{\boldsymbol m\leq \bsnu\\ \boldsymbol m\neq \mathbf 0}}\binom{\bsnu}{\boldsymbol m}((|\boldsymbol m|+k)!)^{\beta}(C\boldsymbol b)^{\boldsymbol m}\Lambda_{\bsnu-\boldsymbol m}+C((|\bsnu|+k)!)^{\beta}(C\boldsymbol b)^{\bsnu},
\end{align*}
where $C_0,C>0$ are constants and $k\in\mathbb N$. Then there holds
$$
\Lambda_{\bsnu}\leq (1+C_0)(C\boldsymbol b)^{\bsnu}(|\bsnu|!)^{\beta}\widetilde\tau_{|\bsnu|,\beta,k},
$$
where 
\begin{align}
\widetilde \tau_{0,\beta,k}=1\quad\text{and}\quad \widetilde \tau_{\nu,\beta,k}=C\sum_{\ell=0}^{v-1}\frac{((\nu-\ell+k)!)^{\beta}}{((\nu-\ell)!)^{\beta}}\widetilde \tau_{\ell,\beta,k}.\label{eq:tautildeseq}
\end{align}
\end{lemma}
\begin{proof}By the induction hypothesis, we obtain
\begin{align*}
\Lambda_{\bsnu}&\leq \!C\!\sum_{\substack{\boldsymbol m\leq \bsnu\\ \boldsymbol m\neq \mathbf 0\\ \boldsymbol m\neq \bsnu}}\!\binom{\bsnu}{\boldsymbol m}((|\boldsymbol m|\!+\!k)!)^{\beta}(C\boldsymbol b)^{\boldsymbol m}(1\!+\!C_0)(C\boldsymbol b)^{\boldsymbol \nu\!-\!\boldsymbol m}((|\bsnu|\!-\!|\boldsymbol m|)!)^{\beta}\widetilde \tau_{|\boldsymbol\nu|\!-\!|\boldsymbol m|,\beta,k}\\
&\quad +(C_0+1)C((|\bsnu|+k)!)^{\beta}(C\boldsymbol b)^{\bsnu}\\
&=(C_0+1)C(C\boldsymbol b)^{\bsnu}\sum_{\ell=1}^{|\bsnu|-1}((\ell+k)!)^{\beta}((|\bsnu|-\ell)!)^{\beta}\widetilde\tau_{|\bsnu|-\ell,\beta,k}\sum_{\substack{|\boldsymbol m|=\ell\\ \boldsymbol m\leq\bsnu}}\binom{\bsnu}{\boldsymbol m}\\
&\quad +(C_0+1)C((|\bsnu|+k)!)^{\beta}(C\boldsymbol b)^{\bsnu}\\
&=(C_0+1)C(C\boldsymbol b)^{\bsnu}\sum_{\ell=1}^{|\bsnu|}((\ell+k)!)^{\beta}((|\bsnu|-\ell)!)^{\beta}\widetilde\tau_{|\bsnu|-\ell,\beta,k}\sum_{\substack{|\boldsymbol m|=\ell\\ \boldsymbol m\leq\bsnu}}\binom{\bsnu}{\boldsymbol m}\\
&\leq (C_0+1)C(C\boldsymbol b)^{\bsnu}(|\bsnu|!)^{\beta}\sum_{\ell=1}^{|\bsnu|}\frac{((\ell+k)!)^{\beta}}{(\ell!)^{\beta}}\widetilde\tau_{|\bsnu|-\ell,\beta,k}\\
&=(C_0+1)C(C\boldsymbol b)^{\bsnu}(|\bsnu|!)^{\beta}\sum_{\ell=0}^{|\bsnu|-1}\frac{((|\bsnu|-\ell+k)!)^{\beta}}{(|\bsnu|-\ell!)^{\beta}}\widetilde\tau_{\ell,\beta,k}\\
&=(C_0+1)(C\boldsymbol b)^{\bsnu}(|\bsnu|!)^{\beta}\widetilde\tau_{|\bsnu|,\beta,k},
\end{align*}
as desired.\end{proof}
\begin{lemma}\label{L2reg}
Under Assumptions \ref{a1}\,--\,\ref{a2sv}, \ref{a4} as well as \ref{A10} the solution $\widehat u$ of the variational problem
 \begin{align*}
b(\bsy;\widehat u,\widehat v)=F(\bsy;\widehat v)\quad\text{for all}~\widehat v=(\widehat v_1,\widehat v_2)\in\mathcal Y~\text{and}~\bsy\in U,
\end{align*}
where 
\begin{align*}
b(\bsy;\widehat u,\widehat v):=&\int_I \langle\tfrac{\partial}{\partial t}\widehat u(\cdot,t,\bsy),\widehat v_1(\cdot,t)\det J(\cdot,\bsy)\rangle_{H^{-1},H_0^1}\,{\rm d}t\\
&\quad +\int_I\int_{{D_{\rm ref}}} A(\bsx,\bsy)\nabla \widehat u(\bsx,t,\bsy)\cdot \nabla \widehat v_1(\bsx,t)\,{\rm d}\bsx\,{\rm d}t\\
&\quad+\int_{D_{\rm ref}} \widehat u(\bsx,0,\bsy)\widehat v_2(\bsx)\det J(\bsx,\bsy)\,{\rm d}\bsx
\end{align*}
and
\begin{align*}
F(\bsy;\widehat v):=&\int_I \int_{D_{\rm ref}}  f_{\rm ref}(\bsx,\bsy)\widehat v_1(\bsx,t)\,{\rm d}\bsx\,{\rm d}t\\
&\quad+\int_{D_{\rm ref}} \widehat u_0(\bsx,\bsy)\widehat v_2(\bsx)\det J(\bsx,\bsy)\,{\rm d}\bsx
\end{align*}
has the property
\begin{align*}
  \partial_\bsy^\bsnu \widehat{u}(\cdot,\cdot,\bsy)\in L^2(I;H_0^1(D_\mathrm{ref}))
\end{align*}
for all $\bsnu\in\mathscr F$.
\end{lemma}
\begin{proof}
    We will perform a proof by induction over the differentiation order $\bsnu$ and show that the parametric derivatives lie in the space $\mathcal X$.

    For $\bsnu=\boldsymbol{0}$ we have that $\widehat{u}(\cdot,\cdot,\bsy)\in\mathcal X$, which holds true. We fix $\bsnu\in\mathscr F\setminus\{\boldsymbol{0}\}$ and assume that the claim holds true for all multi-indices with order less than $|\bsnu|$, which means that all the lower order parametric derivatives of $\widehat{u}$ lie in $\mathcal X$.
    
    In order to find the correct operator equation later used in the proof, we formally differentiate the equation 
\begin{align*}
b(\bsy;\widehat{u},\widehat{v})=F(\bsy;\widehat{v})
\end{align*}
with respect to $\bsy\in U$, where $\widehat{v}=(\widehat{v}_1,\widehat{v}_2)\in\mathcal Y$, i.e.,
\begin{align*}
    \partial_\bsy^\bsnu b(\bsy;\widehat{u},\widehat{v})=\partial_\bsy^\bsnu F(\bsy;\widehat{v}).
\end{align*}
By Leibniz product rule we get
\begin{align*}
&\sum_{\boldsymbol m\leq\bsnu}\binom{\bsnu}{\boldsymbol m}\bigg[\int_I \langle\tfrac{\partial}{\partial t}(\partial_\bsy^{\bsnu-\boldsymbol m}\widehat u(\cdot,t,\bsy)),\widehat v_1(\cdot,t)\partial_\bsy^{\boldsymbol m}\det J(\cdot,\bsy)\rangle_{H^{-1},H_0^1}\,{\rm d}t\\
&\quad +\int_I\int_{D_{\rm ref}}\partial_\bsy^{\boldsymbol m}A(\bsx,\bsy)\nabla \partial_{\bsy}^{\bsnu-\boldsymbol m}\widehat u(\cdot,t,\bsy)\cdot \nabla \widehat v_1(\bsx,t)\,{\rm d}\bsx\,{\rm d}t\\
&\quad +\int_{D_{\rm ref}}\partial_\bsy^{\bsnu-\boldsymbol m}\widehat u(\bsx,0,\bsy)\widehat v_2(\bsx)\partial_{\bsy}^{\boldsymbol m}\det J(\bsx,\bsy)\,{\rm d}\bsx\bigg]\\
&=\int_I \int_{D_{\rm ref}}\partial_{\bsy}^{\bsnu}f_{\rm ref}(\cdot,\bsy)\widehat v_1\,{\rm d}\bsx\,{\rm d}t\\
&\quad + \sum_{\boldsymbol m\leq \bsnu}\binom{\bsnu}{\boldsymbol m}\int_{D_{\rm ref}}\partial_{\bsy}^{\bsnu-\boldsymbol m}\widehat u_0(\bsx,\bsy)\widehat v_2(\bsx)\partial_{\bsy}^{\boldsymbol m}\det J(\bsx,\bsy)\,{\rm d}\bsx,
\end{align*}
where $\widehat u_0(\bsx,\bsy)=u_0(\boldsymbol V(\bsx,\bsy),\bsy)$ for $\bsx\in D_{\rm ref}$, $\bsy\in U$.

Isolating the case $\bsm=\boldsymbol{0}$ we can rewrite this as
\begin{align*}
&\int_I \langle\tfrac{\partial}{\partial t}(\partial_\bsy^{\boldsymbol \nu}\widehat u(\cdot,t,\bsy)),\widehat v_1(\cdot,t)\det J(\cdot,\bsy)\rangle_{H^{-1},H_0^1}\,{\rm d}t\\
&\quad +\int_I\int_{D_{\rm ref}}A(\bsx,\bsy)\nabla\partial_\bsy^{\boldsymbol \nu} \widehat u(\bsx,t,\bsy)\cdot \nabla \widehat v_1(\bsx,t)\,{\rm d}\bsx\,{\rm d}t\\
&\quad +\int_{D_{\rm ref}}\partial_\bsy^{\boldsymbol \nu}\widehat u(\bsx,0,\bsy)\widehat v_2(\bsx)\det J(\bsx,\bsy)\,{\rm d}\bsx\\
&\quad +\sum_{\boldsymbol{0}\neq\boldsymbol m\leq\bsnu}\binom{\bsnu}{\boldsymbol m}\bigg[\int_I \langle\tfrac{\partial}{\partial t}(\partial_\bsy^{\bsnu-\boldsymbol m}\widehat u(\cdot,t,\bsy)),\widehat v_1(\cdot,t)\partial_\bsy^{\boldsymbol m}\det J(\cdot,\bsy)\rangle_{H^{-1},H_0^1}\,{\rm d}t\\
&\quad +\int_I\int_{D_{\rm ref}}\partial_\bsy^{\boldsymbol m}A(\bsx,\bsy)\nabla \partial_{\bsy}^{\bsnu-\boldsymbol m}\widehat u(\bsx,t,\bsy)\cdot \nabla \widehat v_1(\bsx,t)\,{\rm d}\bsx\,{\rm d}t\\
&\quad +\int_{D_{\rm ref}}\partial_\bsy^{\bsnu-\boldsymbol m}\widehat u(\bsx,0,\bsy)\widehat v_2(\bsx)\partial_{\bsy}^{\boldsymbol m}\det J(\bsx,\bsy)\,{\rm d}\bsx\bigg]\\
&=\int_I \int_{D_{\rm ref}}\partial_{\bsy}^{\bsnu}f_{\rm ref}(\bsx,\bsy)\widehat v_1(\bsx)\,{\rm d}\bsx\,{\rm d}t\\
&\quad + \sum_{\boldsymbol m\leq \bsnu}\binom{\bsnu}{\boldsymbol m}\int_{D_{\rm ref}}\partial_{\bsy}^{\bsnu-\boldsymbol m}\widehat u_0(\bsx,\bsy)\widehat v_2(\bsx)\partial_{\bsy}^{\boldsymbol m}\det J(\bsx,\bsy)\,{\rm d}\bsx,
\end{align*}
Denoting
\begin{align*}
    T_\bsm&:=\int_I \langle\tfrac{\partial}{\partial t}(\partial_\bsy^{\bsnu-\boldsymbol m}\widehat u(\cdot,t,\bsy))\widehat v_1(\cdot,t)\partial_\bsy^{\boldsymbol m}\det J(\cdot,\bsy)\rangle_{H^{-1},H_0^1}\,{\rm d}t\\
&\quad +\int_I\int_{D_{\rm ref}}\partial_\bsy^{\boldsymbol m}A(\bsx,\bsy)\nabla \partial_{\bsy}^{\bsnu-\boldsymbol m}\widehat u(\bsx,t,\bsy)\cdot \nabla \widehat v_1(\bsx,t)\,{\rm d}\bsx\,{\rm d}t\\
&\quad +\int_{D_{\rm ref}}\partial_\bsy^{\bsnu-\boldsymbol m}\widehat u(\bsx,0,\bsy)\widehat v_2(\bsx)\partial_{\bsy}^{\boldsymbol m}\det J(\bsx,\bsy)\,{\rm d}\bsx
\end{align*}
we can write
\begin{align}
b(\bsy;\partial_\bsy^\bsnu\widehat{u},\widehat{v})=\partial_\bsy^\bsnu F(\bsy;\widehat{v})-\sum_{\boldsymbol{0}\neq\boldsymbol m\leq\bsnu}\binom{\bsnu}{\boldsymbol m}T_\bsm(\bsy;\partial_\bsy^{\bsnu-\bsm}\widehat{u},\widehat{v}).\label{IDbilinearform}\enspace 
\end{align}
Note that these terms are well-defined due to the induction hypothesis.

Now define a functional $R_\bsnu(\bsy)\in\mathcal Y^*$ by
\begin{align*}
    \langle R_\bsnu(\bsy),\widehat{v}\rangle:=\partial_\bsy^\bsnu F(\bsy;\widehat{v})-\sum_{\boldsymbol{0}\neq\boldsymbol m\leq\bsnu}\binom{\bsnu}{\boldsymbol m}T_\bsm(\bsy;\partial_\bsy^{\bsnu-\bsm}\widehat u,\widehat{v})
\end{align*}
as well as $\mathcal B(\bsy):\mathcal{X}\rightarrow \mathcal Y^*$ by
\begin{align*}
    \langle \mathcal B(\bsy)w,\widehat{v}\rangle=b(\bsy;w,\widehat{v}),\enspace w\in\mathcal X.
\end{align*}
Then, the above identity~\eqref{IDbilinearform} can be expressed as 
\begin{align*}
    \mathcal B(\bsy)w=R_\bsnu(\bsy).
\end{align*}
For this operator equation to make sense we note that for $w\in\mathcal X$ and $\widehat v\in\mathcal Y$ there holds
\begin{align}
    |T_\bsm(\bsy;w,\widehat v)|&\leq \int_I |\langle\tfrac{\partial}{\partial t}w\partial_\bsy^\bsm\det J(\cdot,\bsy),\widehat v_1(\cdot,t)\rangle_{H^{-1},H_0^1}|\,{\rm d}t\notag\\
&\quad +\int_I\int_{D_{\rm ref}}|\partial_\bsy^{\boldsymbol m}A(\bsx,\bsy)\nabla w\cdot \nabla \widehat v_1(\bsx,t)|\,{\rm d}\bsx\,{\rm d}t \notag\\
&\quad +\int_{D_{\rm ref}}|w(\cdot,0)\widehat v_2(\bsx)\partial_{\bsy}^{\boldsymbol m}\det J(\bsx,\bsy)|\,{\rm d}\bsx \notag\\
&\leq C\|\partial_\bsy^\bsm \det J(\cdot,\bsy)\|_{W^{1,\infty}}\|\tfrac{\partial}{\partial t}w\|_{L^2(I;H^{-1})}\|\widehat{v}_1\|_{L^2(I;H_0^1)} \notag \\
&\quad +\|\partial_\bsy^{\boldsymbol m}A(\cdot,\bsy)\|_{L^\infty}\int_I\int_{D_{\rm ref}}|\nabla w\cdot \nabla \widehat v_1(\bsx,t)|\,{\rm d}\bsx\,{\rm d}t\notag\\
&\quad +\|\partial_\bsy^\bsm\det J(\cdot,\bsy)\|_{L^\infty}\int_{D_{\rm ref}}|w(\cdot,0)\widehat v_2(\bsx)|\,{\rm d}\bsx.\label{eq:Operatorbound}
\end{align}
This observation and applying Cauchy--Schwarz inequality to the other two summands leads to a bound of the expression~\eqref{eq:Operatorbound} given by
\begin{align*}
&C \|\partial_\bsy^\bsm \det J(\cdot,\bsy)\|_{W^{1,\infty}} \|\tfrac{\partial}{\partial t}w\|_{L^2(I;H^{-1})}\|\widehat{v}_1\|_{L^2(I;H_0^1)}\\
&\quad +\|\partial_\bsy^{\boldsymbol m}A(\cdot,\bsy)\|_{L^\infty}\|\nabla w\|_{L^2(D\times I)}\|\nabla \widehat{v}_1\|_{L^2(D\times I)}\\
&\quad + \|\partial_\bsy^\bsm \det J(\cdot,\bsy)\|_{L^\infty} \|w(0)\|_{L^2}\|\widehat{v}_2\|_{L^2}\\
&=C\|\partial_\bsy^\bsm \det J(\cdot,\bsy)\|_{W^{1,\infty}}\\
&\quad\times (\|\tfrac{\partial}{\partial t}w\|_{L^2(I;H^{-1})}\|\widehat{v}_1\|_{L^2(I;H_0^1)}+\|w(0)\|_{L^2}\|\widehat{v}_2\|_{L^2})\\
&\quad+\|\partial_\bsy^\bsm A(\cdot,\bsy)\|_{L^\infty}\|\nabla w\|_{L^2(D\times I)}\|\nabla v_1\|_{L^2(D\times I)}\\
&\leq (1+C_{\mathrm{tr}})C\|\partial_\bsy^\bsm \det J(\cdot,\bsy)\|_{W^{1,\infty}}\|w\|_{\mathcal X}\|\widehat v\|_{\mathcal Y}\\
&\quad +\|\partial_\bsy^\bsm A(\cdot,\bsy)\|_{L^\infty}\|w\|_{\mathcal X}\|\widehat v\|_{\mathcal Y}\\
&\leq C(\bsy,\bsm)\|w\|_\mathcal X\|\widehat{v}\|_\mathcal Y,
\end{align*}
where $C(\bsy,\bsm):=(1+C_{\mathrm{tr}})C\|\partial_\bsy^\bsm \det J(\cdot,\bsy)\|_{W^{1,\infty}}+\|\partial_\bsy^\bsm A(\cdot,\bsy)\|_{L^\infty}$.
Note that in the penultimate inequality we use the fact that 
\begin{align*}
   \mathcal X\hookrightarrow \mathcal C([0,T], L^2(D_{\rm ref}))
\end{align*}
 and hence $\|w(0)\|_{L^2}\leq C_{\mathrm{tr}}\|w\|_{\mathcal X}$ (cf.~{\citet[Theorem~5.9.3]{evans}}).
It follows that $T_\bsm\in\mathcal Y^*$. Exploiting the bounds found in the regularity analysis, one easily sees that a similar reasoning reveals that $\partial_\bsy^\bsnu F(\bsy;\cdot)\in\mathcal Y^*$. Thus, $R_\bsnu(\bsy)\in\mathcal Y^*$.

By Assumption~(A2), the Jacobian determinant $\det J(\cdot,\bsy)$ is bounded above and below away from zero and the coefficient matrix
\begin{align*}
A(\bsx,\bsy) = (J(\bsx,\bsy)^{\rm T}J(\bsx,\bsy))^{-1}\det J(\bsx,\bsy)
\end{align*}
is uniformly bounded as well as uniformly elliptic. 
In consequence, the pullback space-time bilinear form
\begin{align*}
b(\bsy;w,\widehat v)
&= \int_I \langle \partial_t w(t), \widehat v_1(t)\rangle_{H^{-1},H_0^1}\,\mathrm{d}t\\
&\quad + \int_I \int_{D_{\mathrm{ref}}} A(\bsx,\bsy)\nabla w\cdot\nabla \widehat v_1 \,\mathrm{d}\bsx\,\mathrm{d}t\\
&\quad+ \int_{D_{\mathrm{ref}}} w(\bsx,0)\widehat v_2(\bsx)\det J(\bsx,\bsy)\,\mathrm{d}\bsx
\end{align*}
satisfies the boundedness and coercivity assumptions of {\citet[Theorem~5.1]{schwabstevenson}}.
Therefore, the induced operator
\[
\mathcal B(\bsy):\mathcal X \to \mathcal Y^*, 
\qquad
\langle \mathcal B(\bsy)w,\widehat v\rangle := b(\bsy;w,\widehat v),
\]
is boundedly invertible.

By the bounded invertibility of $\mathcal B(\bsy)$ it follows that there exists a unique solution $w_\bsnu\in\mathcal X$ of the equation $\mathcal B(\bsy)w_\bsnu=R_\bsnu(\bsy).$ By uniqueness and the fact that we have seen that the operator equation is fulfilled by $\partial_\bsy^\bsnu\widehat{u}(\cdot,\cdot,\bsy)$ it follows that $w_\bsnu=\partial_\bsy^\bsnu\widehat{u}(\cdot,\cdot,\bsy)$. Hence, $\partial_\bsy^\bsnu\widehat{u}(\cdot,\cdot,\bsy)\in\mathcal X$. Since $\mathcal X\subset L^2(I;H_0^1(D_\mathrm{ref}))$ the statement is verified.
\end{proof}


\bibliographystyle{abbrvnat}
\bibliography{main}

\end{document}